\documentclass[reqno]{amsart}
\usepackage{a4wide}
\usepackage{graphics,graphicx}
\usepackage[english]{babel}
\usepackage{color}
\usepackage{amsfonts,amsthm,amssymb,amsmath,mathrsfs}
\usepackage{mathabx}

\usepackage[colorlinks=false, linktocpage=true]{hyperref}
\newtheorem{theorem}{Theorem}[section]
\newtheorem{corollary}[theorem]{Corollary}
\newtheorem{lemma}[theorem]{Lemma}
\newtheorem{proposition}[theorem]{Proposition}
\theoremstyle{definition}
\newtheorem{definition}[theorem]{Definition}

\theoremstyle{remark}
\newtheorem{remark}[theorem]{Remark}
\numberwithin{equation}{section}

\newcommand{\N}{\mathbb{N}}

\newcommand{\R}{\mathbb{R}}

\begin{document}
	\title[Blow-up for a cubic NLS type system]{BLOW-UP OF RADIALLY SYMMETRIC SOLUTIONS FOR A CUBIC NLS TYPE SYSTEM IN DIMENSION 4}

	\author{MAICON HESPANHA}\address{Departamento do Matemática, Instituto de Matemática, Estatística e Computação Científica (IMECC), Universidade Estadual de Campinas (UNICAMP) \\
		Rua S\'ergio Buarque de Holanda, 651,  13083-859, Campinas--SP, Brazil}
	\email{mshespanha@gmail.com}
	
		\author{ADEMIR PASTOR}\address{Departamento do Matemática, Instituto de Matemática, Estatística e Computação Científica (IMECC), Universidade Estadual de Campinas (UNICAMP) \\
			Rua S\'ergio Buarque de Holanda, 651,  13083-859, Campinas--SP, Brazil}
		\email{apastor@ime.unicamp.br}


\begin{abstract} 
This paper is concerned with a cubic nonlinear Schor\"odinger system modeling the interaction between an optical beam and its third harmonic in a material with Kerr-type nonlinear response. We are mainly interested in the so-called energy-critical case, that is, in dimension four. Our main result states that radially symmetric solutions with initial energy below that of the ground states but with kinetic energy above that of the ground states must blow-up in finite time. The proof of this result is based on the convexity method. As an independent interest we also establish the existence of ground state solutions, that is, solutions that minimize some action functional. In order to obtain our existence results we use the concentration-compactness method combined with variational arguments. As a byproduct, we also obtain the best constant in a vector critical Sobolev-type inequality.
\end{abstract}

\maketitle

\section{Introduction}\label{intr}
We study the following nonlinear Schrödinger system 
\begin{equation}\label{SIST1}
\begin{cases}
iu_t+\Delta u-u+\left(\dfrac{1}{9}|u|^2+2|w|^2\right)u+\dfrac{1}{3}\bar{u}^2w=0,\\
i\sigma w_t+\Delta w-\mu w+(9|w|^2+2|u|^2)w+\dfrac{1}{9}u^3=0,\\
\end{cases}
\end{equation}
where $u=u(t,x)$ and $w=w(t,x)$ are complex-valued functions with $(t,x)\in \mathbb{R}\times \mathbb{R}^d$, $\Delta$ represents the standard Laplacian operator and $\sigma,\mu$ are positive real constants. This model describes the interaction between an optical beam and its third harmonic in a material with Kerr-type nonlinear response. For a more detailed explanation of the model, the reader can check \cite{SBK}.

From a mathematical point of view,  system \eqref{SIST1} has been studied in several situations. In \cite{AP} and \cite{pastor4}, the authors established local and global well-posedness for the associated Cauchy problem with periodic initial data in dimension one; the existence of periodic standing waves of dnoidal and cnoidal-type as well as their nonlinear and spectral stability were also established in the energy space. More precisely, the authors proved the existence of two smooth curves of periodic solutions and established the stability/instability  under several perturbations regime. For the multidimensional case in $\R^d$, $1\leq d\leq3$, several results on the dynamics of \eqref{SIST1} were proved in \cite{oliveira};  such results include existence and stability of ground state solutions, local and global well-posedness in the energy space $H^1(\mathbb{R}^d)\times H^1(\mathbb{R}^d)$  and several criteria for blow-up in finite time. To be more precise, we recall that \eqref{SIST1} conserves the energy and the mass given, respectively, by
\begin{equation}\label{Energia}
	E(u,w):=\frac{1}{2}\int (|\nabla u|^2 +|\nabla w|^2+|u|^2+\mu|w|^2)-\int\left(\frac{1}{36}|u|^4+\frac{9}{4}|w|^4+|u|^2|w|^2+\frac{1}{9}\hbox{Re}(\bar{u}^3w)\right)
\end{equation}
and
\begin{equation}\label{M}
	M(u,w):=\int(|u|^2+3\sigma|w|^2),
\end{equation}
which means that such a quantities do not depend on time when evaluated at a solution of \eqref{SIST1}. The of local well-posedness for the Cauchy problem associated with \eqref{SIST1} in the energy space  is standard by now and it follows as in the scalar cubic nonlinear Schr\"odinger equation taking into account the well-known Strichartz estimates. Indeed, we can prove the local well-posedness using the fixed point method to find solutions of the equivalent integral equations  
\begin{equation*}
\begin{cases}
	{\displaystyle	u(t)=U(t)u_0+i\int_0^t U(t-s)F(u(s),w(s))ds,}\\
	{\displaystyle	w(t)=W(t)w_0+i\int_0^t W(t-s)G(u(s),w(s))ds,}\\
	\end{cases}\
\end{equation*}
where $U(t)=e^{it(-\Delta +1)}$ and $W(t)=e^{it(-a\Delta +b)}$ are the corresponding unitary groups associated to the linear part of \eqref{SIST1}, with $a=1/\sigma$, $b=\mu/\sigma$, and 
\begin{equation*}
	F(u,w)= \left(\frac{1}{9}|u|^2+2|w|^2\right)u + \frac{1}{3}\bar{u}^2w \quad \mbox{and}\quad G(u,w)=a(9|u|^2 + 2|w|^2)w +\frac{1}{9}au^3.
\end{equation*}
On the other hand, the extension of the local solution to a global one depends on the dimension. It is well-known that the cubic nonlinearity is (mass) subcritical, critical, and supercritical, with respect to scaling, in dimensions $d=1$, $d=2$, and $d=3$, respectively. With this in mind, in the case $d=1$, using conservation of the energy and the mass, the local solution can be promptly extended to a global one. In the case $d=2$, we obtain global solutions provided the initial data satisfies $M(u_0,w_0)<M(P,Q)$, where $(P,Q)$ is any  ground state solution of the elliptic system (with $\omega=0$ and $\mu=3\sigma$)
\begin{equation}\label{SIST2}
	\begin{cases}
		{\displaystyle	\Delta P-(\omega+1)P+\left(\frac{1}{9}P^2+2Q^2\right)P+\frac{1}{3}P^2Q=0,}\\
		{\displaystyle \Delta Q-(\mu+3\sigma\omega) Q+(9Q^2+2P^2)Q+\frac{1}{9}P^3=0.}
	\end{cases}
\end{equation}
An (action) ground state is a solution of \eqref{SIST2} that minimizes the action $E+\frac{\omega}{2}M$. In \cite{oliveira} it was proved that such a minimizers do exist provided $\omega>\max\{-1,-\mu/3\sigma\}$.
In addition, in the case $d=3$, the local solution can be extended to a global one provided $2E(u_0,w_0)M(u_0,w_0)<E(P,Q)M(P,Q)$ and $K(u_0,w_0)M(u_0,w_0)<K(P,Q)M(P,Q)$, where
$$
K(u,w)=\int |\nabla u|^2+|\nabla w|^2.
$$
The details of these results may be found in \cite[Section 3]{oliveira}. We also recall that sharp threshold criteria to global well-posedness and scattering of solutions and the formation of
singularities in finite time for (anisotropic) symmetric initial data were established in \cite{ardila}.

Recall that system \eqref{SIST2} appears when we are looking for  standing-wave  solutions of \eqref{SIST1} of the form
\begin{equation*}
	u(x,t)=e^{i\omega t}P(x),\quad w(x,t)=e^{3i\omega t}Q(x),
\end{equation*}
where $P$ and $Q$ are real-valued functions with a suitable decay at infinity. The orbital stability/instability of the ground states was also studied in \cite{oliveira} under some suitable conditions on the parameters. In particular the orbital stability was proved under the assumption $d=1$ and $\omega+1=\mu+3\sigma$. Also, the orbital instability was established provided $d=3$ and $\mu>0$  or $d=2$ and $\mu\neq3\sigma$. More recently, these results were extended in \cite{stefanov}, where the authors constructed the waves in the largest possible parameter regime and established a complete classification of their spectral stability/instability. The instability by blow-up was also established in dimensions 2 and 3, depending on the parameters $\mu$ and $\sigma$. Also, in \cite{colin}, it was obtained (energy) ground state solutions with a prescribed mass; in particular, their main result establishes the relation between energy ground states, action ground states and minimizers of the well-known  Weinstein functional. Finally, in \cite{zhang}, it was proved existence results for normalized ground state solutions in the $L^2$-subcritical and $L^2$-supercritical cases and  the nonexistence of normalized ground state  in the $L^2$-critical case and a new
blow-up criterion which is related to normalized solutions. For additional results concerning the existence of solitary traveling-wave solutions, we refer the reader to \cite{likai}.

In the present paper, we are interested in showing the existence of finite time blowing-up solutions in the energy-critical case, namely, in the case $d=4$. Before describing our results, we recall the results for the cubic Schr\"odinger equation
\begin{equation}\label{NLS}
	iu_t+\Delta u + |u|^2u=0,
\end{equation}
in $\dot{H}^1(\mathbb{R}^4)$. The Cauchy problem associated with \eqref{NLS} was first studied in \cite{CAZW}, where the authors established a local well-posedness result. The study of global well-posedness in the radial case was developed in the seminal paper by Kenig and Merle \cite{KM}. To be more precise, let us define the function
\begin{equation}\label{solitonNLS}
	W(x)=\left(1+\frac{|x|^2}{8}\right)^{-1}.
\end{equation}
It is well-known that $W$ is the ground state solution of the elliptic equation
\begin{equation*}
	\Delta W +|W|^2W=0.
\end{equation*}
In addition, it belongs to $\dot{H}^1(\mathbb{R}^4)$ and is connected with the best constant in the critical Sobolev inequality  (see \cite{T}). The results in \cite{KM} show that if a radial initial data satisfies $\tilde{E}(u_0)<\tilde{E}(W)$ and $\|\nabla u_0\|_{L^2}<\|\nabla W\|_{L^2}$ (here $\tilde{E}$ is the energy functional associated with \eqref{NLS}) then the corresponding solution is global and scatters; this means that it behaves like a solution of the corresponding linear problem at infinity. On the other hand, if a radial initial data is such that $|x|u_0\in L^2(\mathbb{R}^4)$ and satisfies $\tilde{E}(u_0)<\tilde{E}(W)$ and $\|\nabla u_0\|_{L^2}>\|\nabla W\|_{L^2}$ then the solution blows-up in finite time. Actually, the results in \cite{KM} were proved for the general energy-critical Schr\"odinger equation (power $4/(d-2)$ instead of $2$ in \eqref{NLS})  in dimensions $3\leq d\leq 5$. These results were extended,  without the assumption of radial initial data, to dimensions $d\geq 5$ in \cite{KV} and more recently to dimension $d=4$ in \cite{D}.

By following the strategy for the Schr\"odinger equation \eqref{NLS} (see, for instance, \cite[Chapter 4]{cazenave} or \cite[Chapter 5]{linares}), we may establish the local well-posedness of \eqref{SIST1} in the energy space. Indeed, define (see next section for notations)
\begin{equation*}
	Y(I):=C(I; H^1(\R^4))\cap L^4(I;H^{1,8/3}(\R^4)),
\end{equation*}
for a time interval $I=[-T,T]$ with $T>0$. The result is the following.
\begin{theorem}\label{BCL}
	For any $u_0,w_0\in H^1(\mathbb{R}^4)$, there exists $T(u_0,w_0)>0$, such that system \eqref{SIST1} has a unique solution $(u,w)\in Y(I)\times Y(I)$, with $I=[-T(u_0,w_0),T(u_0,w_0)]$. In addition, the map data-solution is continuous and the following blow-up alternative holds: There exist times $T_*,T^*\in(0,\infty]$ such that the solution can be extended to $(-T_*,T^*)$ and if $T^*<\infty$, then
	\begin{equation*}
		\Vert \nabla u(t)\Vert_{L_t^q([0,T^*];L_x^{r})}+\Vert\nabla w(t)\Vert_{L_t^q([0,T^*];L_x^{r})}=\infty,
	\end{equation*}
	for any pair $(q,r)$ satisfying $2< q< \infty$, $2<  r< 4$, and $ \displaystyle\frac{2}{q}=2-\frac{4}{r}$. 
 A similar result holds if $T_*<\infty.$
\end{theorem}

In order to establish our blow-up results, we need to understand the analogue of \eqref{solitonNLS} for system \eqref{SIST1}. Let us start by recalling, from \cite[Lemma 2.2]{oliveira}, that if $(P,Q)\in H^1(\mathbb{R}^d)\times H^1(\mathbb{R}^d)$ is a solution of \eqref{SIST2} then it holds 
\begin{equation*}\label{L22}
	(d-4)\int (|\nabla P|^2+|\nabla Q|^2)+d(\omega+1)\int P^2+d(\mu+3\sigma\omega)\int Q^2=0.
\end{equation*}
Hence, if $d=4$ we have
$$
(\omega+1)\int P^2+(\mu+3\sigma\omega)\int Q^2=0,
$$
and we expect the existence of non-trivial solution for $\omega=-1$ and $\mu=3\sigma$. In this case, system \eqref{SIST2} reduces to
\begin{equation}\label{SIST3}
	\begin{cases}
		{\displaystyle 	\Delta P+\left(\frac{1}{9}P^2+2Q^2\right)P+\frac{1}{3}P^2Q=0,}\\
		{\displaystyle \Delta Q+(9Q^2+2P^2)Q+\frac{1}{9}P^3=0.}
	\end{cases}
\end{equation}
Thus our first task is to show the existence of solutions for the above elliptic system, which may be seen as critical points of the action
\begin{equation*}
	S(P,Q)=\frac{1}{2}K(P,Q)-N(P,Q).
\end{equation*}
where 
\begin{equation}\label{K&N}
	N(P,Q)=\int \frac{1}{36}P^4+\frac{9}{4}Q^4+P^2Q^2+\frac{1}{9} P^3Q.
\end{equation}

Precisely, 
\begin{definition}\label{gsdef}
	We say that
	\begin{itemize}
		\item[(i)] a pair of functions $(P,Q)\in \dot{H}^1(\mathbb{R}^4)\times\dot{H}^1(\mathbb{R}^4)$ is a weak solution to \eqref{SIST3}, if for all  $(f,g)\in \dot{H}^1(\mathbb{R}^4)\times\dot{H}^1(\mathbb{R}^4)$,
		\begin{equation}\label{solfraca}
			\begin{split}
				&\int \nabla P\cdot \nabla fdx=\int \left(\frac{1}{9}P^3+2Q^2P+\frac{1}{3}P^2Q\right)fdx,\\
				&\int \nabla Q\cdot \nabla gdx=\int \left(9Q^3+2P^2Q+\frac{1}{9}P^3\right)gdx.
			\end{split}
		\end{equation}
		\item[(ii)] A solution $(P_0,Q_0)\in \dot{H}^1(\mathbb{R}^4)\times\dot{H}^1(\mathbb{R}^4)$ is a \textit{ground state} of \eqref{SIST3} if
		$$
		S(P_0,Q_0)=\inf\{S(P,Q);\,\, (P,Q)\in\mathcal{C}\}
		$$
		where $\mathcal{C}$ denotes the set of all non-trivial solutions of  \eqref{SIST3}.  The set of all ground states  will be  denote by $\mathcal{G}$.
	\end{itemize}
\end{definition}

Our main result concerning ground states is the following.
\begin{theorem}\label{ESGS}
	There exists at least one ground state solution $(P_0,Q_0)$ for system \eqref{SIST3}, i.e., $\mathcal{G}$ is non-empty.
\end{theorem}
To prove Theorem \ref{ESGS}, we shall employ the concentration-compactness method, introduced in \cite{lions2}, which says that any sequence of probability measures must have a subsequence such that either vanishing or dichotomy or compactness (see Lemma \ref{lema35}) must occur. As we will see, the ground states may be seen as minimizers of a normalized minimization problem. In order to find at least one minimizer of this problem we shall construct a suitable sequence of probability Radon measures. After some calculations, we will avoid  vanishing and dichotomy, implying in the compactness of the sequence.  Up to dilatation and translation such a sequence converges to the desired minimizer. As a consequence of Theorem \ref{ESGS}, we also obtain the optimal constant for the critical Sobolev inequality.

In what blow-up in finite time is concerned, our result reads as follows.
\begin{theorem}\label{blow4}
	Suppose $(u_0,w_0)\in H^1(\mathbb{R}^4)\times H^1(\mathbb{R}^4)$ and let $(u,w)$ be the corresponding solution of \eqref{SIST1} defined in the maximal time interval of existence, say, $I$. If $(u_0,w_0)$  is a pair of radially symmetric functions satisfying
	\begin{equation*}
		E(u_0,w_0)<\mathcal{E}(P,Q)
	\end{equation*}
	\begin{equation*}
		K(u_0,w_0)>K(P,Q),
	\end{equation*}
	where $(P,Q)$ is any ground state in $\mathcal{G}$ and $\mathcal{E}$ is the energy defined in \eqref{E}, then the time interval $I$ is finite.
\end{theorem} 

As usual, in order to prove Theorem \ref{blow4}, we shall use the convexity method, which, roughly speaking, consists in deriving a contradiction by working with the ``virial function''
$$
\mathcal{V}(t)=\int \phi(x)(|u(t,x)|^2+\sigma^2|w(t,x)|^2)dx,
$$
for some suitable $\phi\in C^\infty(\mathbb{R}^4)$, and its derivative
\begin{equation}\label{V'}
	\mathcal{V}'(t)=2\hbox{Im}\int\nabla\phi\cdot(\nabla u\bar{u}+\sigma\nabla w\bar{w})dx-2\int\phi \hbox{Im}\left(\bar{u}f(u,w)+\sigma\bar{w}g(u,w)\right)dx.
\end{equation}
Here, $f(u,w)=\left(\frac{1}{9}|u|^2+2|w|^2\right)u+\frac{1}{3}\bar{u}^2w$ and $g(u,w)=(9|w|^2+2|u|^2)w+\frac{1}{9}u^3$.

Notice that the second term in \eqref{V'} does not  necessarily vanish (unless $\sigma=3$), which brings some difficulties in order to apply the method. To overcome this, we shall use a modification of the method as in \cite{inui}, which consists in dropping the second term in \eqref{V'} and work with radially symmetric solutions. Thus we work with the function
\begin{equation*}
	\mathcal{R}(t)=2\hbox{Im}\int\nabla\phi\cdot(\nabla u\bar{u}+\sigma\nabla w\bar{w})dx
\end{equation*}
instead of the usual $\mathcal{V}'$.

\begin{remark}
An initial data satisfying the assumptions of Theorem \ref{blow4} may be constructed as follows: fix a pair $(u_0,w_0)$ of smooth radially symmetric functions satisfying $\hbox{Re}(\bar{u}_0^3w_0)>0$ and define
$$
(u_{0\lambda},w_{0\lambda})=\lambda(u_0,w_0), \quad \lambda>0.
$$
It is easily seen that
$$
E(u_{0\lambda},w_{0\lambda})=\lambda^2 a-\lambda^4b
$$
and
$$
K	(u_{0\lambda},w_{0\lambda})=\lambda^2 K(u_0,w_0),
$$
where $a$ and $b$ are positive constants depending on $(u_0,w_0)$.
Therefore, if $\lambda$ is large enough then $	(u_{0\lambda},w_{0\lambda})$ satisfies the assumptions of Theorem 1.4.
\end{remark}

As we already pointed out, the global well-posedness (and scattering) of \eqref{NLS} for arbitrary initial data in $\dot{H}^1(\R^4)$ was established in \cite{D} under the assumptions  $\tilde{E}(u_0)<\tilde{E}(W)$ and $\|\nabla u_0\|_{L^2}<\|\nabla W\|_{L^2}$. Therefore, we expect that if the initial data satisfies the assumptions $E(u_0,w_0)<\mathcal{E}(P,Q)$ and $K(u_0,w_0)<K(P,Q)$, then the solution of \eqref{SIST1} must be global in time (at least under the assumption of resonance $\mu=3\sigma$). This is currently under investigation.

The paper is organized as follows. In Section \ref{notation} we just recall some basic notation and two useful lemmas. The existence of ground states solutions is proved in Section \ref{grounds}. In Section \ref{blowup} we establish the existence of blow-up solutions by proving Theorem \ref{blow4}. Finally, in Section  \ref{appendix} we establish the existence of blowing-up solutions in the case of resonance but without the assumption of  radial initial data.

\section{Notation and two useful lemmas}\label{notation}
Throughout the work we will use the standard notation in PDEs. In particular, $C$ will represent a generic constant which may vary from inequality to inequality. Given two positive number $a$ and $b$, we write $a\lesssim b$ whenever  $a\leq Cb$ for some constant $C>0$; similar for the case $a\gtrsim b$. For a complex number $z\in \mathbb{C}$, Re$\,z$ and Im$\,z$ represents its real and imaginary parts.  Also, $\bar{z}$ denotes the complex conjugate of $z$. 

Given two real parameters $s\in\R$ and $1\leq p\leq \infty$, we denote the standard Sobolev, the homogeneous Sobolev and the Lebesgue spaces by  $H^{s,p}=H^{s,p}(\mathbb{R}^d)$, $\dot{H}^{s,p}=\dot{H}^{s,p}(\R^d)$ and $L^p=L^p(\mathbb{R}^d)$, respectively, with their usual norms. As usual, in the case $p=2$, we simplify nations by setting  $H^s=H^{s,2}$ and $\dot{H}^s=\dot{H}^{s,2}$. Given a time interval $I$ and a Banach space $X$,  the space $L^p(I;X)$  will be endowed with the norm
$$
\Vert f \Vert_{L_t^pX}=\left\| \left\|f(t)\|_X\right\|\right\|_{L^p(I)}.
$$
Integration of a function $f$ on $\R^4$ will be denoted by $\int fdx$ or simply $\int f$, if this will not cause a confusion.

Now, we present an adapted version of the generalized Brezis-Lieb Lemma (see \cite[Theorem 2]{BL}). Let $f:\mathbb{R}^l\rightarrow\mathbb{R}$ be a continuous function such that $f(0,...,0)=0$ and for all $a,b\in\mathbb{R}^l$ and $\epsilon>0$,
\begin{equation*}
	|f(a+b)-f(b)|\leq \epsilon\zeta(a)+\psi_\epsilon(b),
\end{equation*}
where $\zeta$ and $\psi_\epsilon$ are non-negative functions.
\begin{lemma}\label{BreLie} 
	Let $v_m=u_m-u$ be a sequence of measurable functions from $\mathbb{R}^d\rightarrow\mathbb{R}^l$ such that
	\begin{itemize}
		\item[(i)] $v_m\rightarrow 0$ a.e.;
		\item[(ii)] $f(u)\in L^1(\mathbb{R}^d)$;
		\item[(iii)] $\int \zeta(v_m)(x)dx\leq M<\infty$, for some constant $M$, independent of $m$;
		\item[(iv)] $\int \psi_\epsilon(u)(x)dx<\infty$, for any $\epsilon>0$.
	\end{itemize}
	Then, as $m\rightarrow\infty$,
	$$
	\int_{\R^d} |f(u_m)-f(v_m)-f(u)|dx\rightarrow 0.
	$$
\end{lemma}
\begin{proof}
The proof of this result was established in \cite[Theorem 2]{BL} for complex-valued functions. However, an inspection in the proof reveals it also holds for vector-valued functions as stated above.
\end{proof}

Next lemma will be used to show the  blow-up in finite time according to Theorem \ref{blow4}.

\begin{lemma}\label{lema comp1}
	Let $I\subset\mathbb{R}$ be an open interval with $0\in I$, $a\in\R$, $b>0$ and $q>1$. Define $\gamma=(bq)^{-1/(q-1)}$ and $f(r)=a-r+br^q$, for $r>0$. Let $G(t)$ be a non-negative continuous function such that $f\circ G\geq 0$ in $I$. By assuming that $a<\displaystyle\left(1-\displaystyle\frac{1}{q}\displaystyle\right)\gamma$, we have
	\begin{itemize}
		\item[(i)] If $G(0)<\gamma$ then $G(t)<\gamma$, for all $t\in I$;
		\item[(ii)] If $G(0)>\gamma$ then $G(t)>\gamma$, for all $t\in I$.
	\end{itemize}
\end{lemma}
\begin{proof}
The proof follows essentially by continuity. The interested reader will find the details, for instance, in \cite[Lemma 3.1]{pastor2} or \cite[Lemma 5.2]{B}.
\end{proof}

\section{Existence of a  ground state solution}\label{grounds}
This section is devoted to prove the existence of ground state solutions as stated in Theorem \ref{ESGS}. As mentioned before, we will follow the ideas in \cite{pastor3}.

\subsection{Critical Sobolev-type inequality and some remarks}
Let us start with some properties of the solutions of \eqref{SIST3}. The first result states that the function $N$, defined in \eqref{K&N},  must be positive when evaluated at a pair $(P,Q)$ of non-trivial solutions of \eqref{SIST3}.

\begin{lemma}\label{lema31}
	Let $\mathcal{N}:=\{(P,Q)\in \dot{H}^1(\mathbb{R}^4)\times\dot{H}^1(\mathbb{R}^4);\,\, N(P,Q)>0\}$. If $\mathcal{C}$ denotes the set of all non-trivial solutions of  \eqref{SIST3} then  $\mathcal{C}\subset \mathcal{N}$. 
\end{lemma}
\begin{proof}
Assume $(P,Q)\in\mathcal{C}$. By inserting $(f,g)=(P,Q)$ in \eqref{solfraca} we obtain
	$$
	\int |\nabla P|^2=\int \frac{1}{9}P^4+2Q^2P^2+\frac{1}{3}P^3Q,
	$$
	and
	$$
	\int |\nabla Q|^2=\int 9Q^4+2Q^2P^2+\frac{1}{9}P^3Q.
	$$
	By summing both equations, we get
	\begin{equation}\label{36}
		K(P,Q)=\int\frac{1}{9}P^4+9Q^4+4Q^2P^2+\frac{4}{9}P^3Q=4N(P,Q).
	\end{equation}
	Since $(P,Q)$ is non-trivial, it follows that $N(P,Q)>0$ and $(P,Q)\in\mathcal{N}$ as desired.
\end{proof}

Let us introduce the functional
\begin{equation*}
	J(P,Q):=\frac{K(P,Q)^2}{N(P,Q)}, \quad (P,Q)\in\mathcal{N}.
\end{equation*}

\begin{remark}\label{obs32}
(i) Since solutions of \eqref{SIST3} do not belong to $L^2$, the associated ``energy functional'' now reads as
		\begin{equation}\label{E}
			\mathcal{E}(u,w):=\frac{1}{2}K(u,w)-N(u,w),\,\, (u,w)\in \dot{H}^1(\mathbb{R}^4)\times\dot{H}^1(\mathbb{R}^4).
		\end{equation}
In particular, the energy now  coincides with the action function, that is, $\mathcal{E}(P,Q)=S(P,Q).$ 
		
	(ii) Observe that, using \eqref{36}, 
		$$
		S(P,Q)=\frac{1}{2}K(P,Q)-N(P,Q)=N(P,Q).
		$$
		Moreover, 
		$$
		J(P,Q)=\frac{K(P,Q)^2}{N(P,Q)}=16N(P,Q)=16S(P,Q).
		$$
	Hence, a non-trivial solution of \eqref{SIST3} is a ground state if, and only if, it has least energy $\mathcal{E}$ among all solutions if, and only if, it minimizes $J$.
\end{remark}

From now on in this section, we will assume that  $u$ and $w$ are real-valued functions. We start by noticing that from  Hölder and Young's inequality 
\begin{equation*}\label{D1}
	\int u^3w \lesssim\Vert u\Vert_{L^4}^4+\Vert w\Vert_{L^4}^4 \quad \mbox{and}\quad  	\int u^2w^2\lesssim\Vert u\Vert_{L^4}^4+\Vert w\Vert_{L^4}^4,
\end{equation*}
which gives $N(u,w)\lesssim\Vert u\Vert_{L^4}^4+\Vert w\Vert_{L^4}^4$. So, using Sobolev's inequality
$
\Vert f\Vert_{L^4}^4\lesssim\Vert \nabla f\Vert_{L^2}^4
$
we get the ``critical inequality''
\begin{equation}\label{N}
	\begin{split}
		N(u,w)\leq C K(u,w)^2.
	\end{split}
\end{equation}
for some positive constant $C$. Concerned with the optimal constant we can place in \eqref{N}, for functions $(u,w)\in\mathcal{N}$, we see that it is given by
\begin{equation}\label{Min}
	C_{\rm opt}^{-1}:=\inf\{J(u,w);\,(u,w)\in\mathcal{N}\},
\end{equation} 

In view of Remark \ref{obs32}, in order to show the existence of a ground state solution, it suffices to show  this infimum is attained. To this end,  we will consider the normalized version of the problem as follows
\begin{equation}\label{MinNorm}
	I:=\inf\{K(u,w);\, (u,w)\in\mathcal{N},\,\, N(u,w)=1\}.
\end{equation}

\begin{definition}
	A minimizing sequence for \eqref{Min} is a sequence $(u_m,w_m)$ in $\mathcal{N}$ such that\- $J(u_m,w_m)\rightarrow	C_{\rm opt}^{-1}$. In the same way, a minimizing sequence for \eqref{MinNorm} is a sequence $(u_m,w_m)$ in $\mathcal{N}$ such that $N(u_m,w_m)=1$, for all $m$, and $K(u_m,w_m)\rightarrow I$.
\end{definition}

Next, since
$
|\nabla |u||^2\leq |\nabla u|^2
$
and 
  $|\nabla |w||^2\leq |\nabla w|^2$ a.e. (see, for instance, \cite[Theorem 6.17]{lieb}), we deduce $K(|u|,|w|)\leq K(u,w)$. Moreover, 
$$
N(u,w)\leq \int \frac{1}{36}|u|^4+\frac{9}{4}|w|^4+|u|^2|w|^2+\frac{1}{9}|u|^3|w|=N(|u|,|w|).
$$
This implies we may assume, with no loss of generality, that minimizing sequences for both \eqref{Min} and \eqref{MinNorm} are always non-negative.
\begin{lemma}\label{equi}
We have  $C_{\rm opt}=I^{-2}$.
\end{lemma}
 \begin{proof}
	Indeed, set $A:=\{(u,w)\in\mathcal{N},\, N(u,w)=1 \}$. Then, for any $(u,w)\in A$ we have $J(u,w)=K(u,w)^2$ and, hence,  $C_{\rm opt}^{-1}\leq K(u,w)^2$ or equivalently $C_{\rm opt}^{-1/2}\leq K(u,w)$, for all $(u,w)\in A$.This means $C_{\rm opt}^{-1/2}$ is a lower bound for the set $\{K(u,w);\, (u,w)\in \mathcal{N},\, N(u,w)=1\}$. Therefore, $C_{\rm opt}^{-1/2}\leq I$, i.e.,  $I^{-2}\leq C_{\rm opt}$. On the other hand, since $N$ and $K$ are homogeneous of degree 4 and 2, respectively, we have that $J(\lambda u,\lambda w)=J(u,w)$, for any $\lambda>0$. Now, given $\epsilon>0$, let $(u,w)\in\mathcal{N}$ be such that $J(u,w)<C_{\rm opt}^{-1}+\epsilon$ and set $(\tilde{u},\tilde{w}):=N(u,w)^{-1/4}(u,w)$. Then  $N(\tilde{u},\tilde{w})=1$, $J(\tilde{u}, \tilde{w})=J(u,w)$ and 
$$I^2\leq  K(\tilde{u},\tilde{w})^2=J(\tilde{u},\tilde{w})=J(u,w)<C_{\rm opt}^{-1}+\epsilon.$$ 
Since $\epsilon>0$ is arbitrary, we get $C_{\rm opt}\leq I^{-2}$ and the proof is completed.
\end{proof}

\begin{remark}\label{obs34i}
 From the above lemma, \eqref{N} becomes
	\begin{equation}\label{DSCN}
		N(u,w)\leq I^{-2}K(u,w)^2,\quad  (u,w)\in\mathcal{N}.
	\end{equation}
	In addition, if $(u,w)$ is a minimizer for \eqref{MinNorm}, then $K(u,w)=I$ and $N(u,w)=1$, so 
	$$
	J(u,w)=\frac{K(u,w)^2}{N(u,w)}=I^2=C_{\rm opt}^{-1}.
	$$
	Thus, $(u,w)$ is also a minimizer for \eqref{Min}. 
\end{remark}

Before proceeding, it is convenient to set the function
$$\varphi(u,w):= \frac{1}{36}u^4+\frac{9}{4}w^4+u^2w^2+\frac{1}{9}u^3w.$$
Notice that $\varphi$ is homogeneous of degree 4. Also,  for $R>0$ and $y\in\mathbb{R}^4$, by defining  the function $u^{R,y}:=R^{-1}u(R^{-1}(x-y))$, we see that
$N(u^{R,y},w^{R,y})=N(u,w)$ and $K(u^{R,y},w^{R,y})=K(u,w)$.
Thus, the functionals $K$ and $N$ are invariant under the transformation 
\begin{equation}\label{obs34ii}
	(u,w)\mapsto(u^{R,y},w^{R,y})=(R^{-1}u(R^{-1}(x-y)),R^{-1}w(R^{-1}(x-y))).
\end{equation}

We end this section with a localized version of Sobolev's inequality. In what follows, by $B(x,R)$ we denote the ball in $\R^4$ with radius $R>0$ centered at $x$. We first state the following.

\begin{lemma}\label{lema38}
Given any $\delta>0$, we may found a constant $C(\delta)>0$ with the following property: if $0<r<R$ are real constants with $r/R<C(\delta)$ and $x\in \mathbb{R}^4$, then there is a cut-off function $\chi_R^r\in H^{1,\infty}(\mathbb{R}^4)$ such that $\chi_R^r=1$ on $B(x,r)$, $\chi_R^r=0$ outside $B(x,R)$ and
	\begin{equation*}
		K(\chi_R^ru,\chi_R^rw)\leq \int_{B(x,R)}\left(|\nabla u|^2+|\nabla w|^2\right)dy+\delta K(u,w),
	\end{equation*}
	and
	\begin{equation*}
		K((1-\chi_R^r)u,(1-\chi_R^r)w)\leq \int_{\mathbb{R}^4\backslash B(x,r)}\left(|\nabla u|^2+|\nabla w|^2\right) dy+\delta K(u,w),
	\end{equation*}
	for any $(u,w)\in \dot{H}^1(\mathbb{R}^4)\times\dot{H}^1(\mathbb{R}^4)$.
\end{lemma}

\begin{proof}
The proof of this result follows the same strategy as in \cite[Lemma 8]{FM} or in \cite[Lemma 3.8]{pastor3}. So, we omit the details.
\end{proof}

With this in hand, we  state the localized version of the Sobolev inequality as follows.
\begin{corollary}\label{corol39}
	Assume $(u,w)\in\dot{H}^1(\mathbb{R}^4)\times\dot{H}^1(\mathbb{R}^4)$ satisfies $u,w>0$. Fix $\delta>0$ and $r/R \leq C(\delta)$ with $C(\delta)$ as in Lemma \ref{lema38}. Then
	\begin{equation*}
		\int_{B(x,r)}\varphi(u,w)dy\leq I^{-2}\left[\int_{B(x,R)}|\nabla u|^2+|\nabla w|^2dy+\delta K(u,w)\right]^{2},
	\end{equation*}
	\begin{equation*}
		\int_{\mathbb{R}^4\backslash B(x,R)}\varphi(u,w)dy\leq I^{-2}\left[\int_{\mathbb{R}^4\backslash B(x,r)}|\nabla u|^2+|\nabla w|^2dy+(2\delta +\delta^2) K(u,w)\right]^{2}.
	\end{equation*}
\end{corollary}
\begin{proof}
This may proved as in \cite[Corollary 9]{FM} or in \cite[Corollary 3.9]{pastor3}. So, we also omit the details here.
\end{proof}

\subsection{Concentration-compactness}

In what follows we shall denote by $\mathcal{M}_+(\R^4)$ the space of all  non-negative Radon measures, by $\mathcal{M}_+^b(\R^4)$ the space of all bounded (or finite) Radon measures and by $\mathcal{M}_+^1(\R^4)$ the space of all probability Radon measures. We write $\nu\ll\mu$  if the measure $\nu$ is absolutely
continuous with respect to the measure $\mu$. Also, by $C_0(\R^4)$ and $C_b(\R^4)$ we denote the space of all compactly supported continuous functions and the space all bounded continuous functions, respectively.

Let us recall some  notions of convergence of measures. 
\begin{definition}
	\begin{itemize}
		\item[(i)] A sequence $(\mu_m)\subset\mathcal{M}_+(\R^4)$ is said to converge vaguely to $\mu$ in $\mathcal{M}_+(\R^4)$, and denoted by $\mu_m\overset{\ast}{\rightharpoonup}\mu$, if $\int fd\mu_m\rightarrow\int fd\mu$ for all $f\in{C}_0(\R^4)$.
		
		\item[(ii)] A sequence $(\mu_m)\subset\mathcal{M}_+^b(\R^4)$ is said to converge weakly to $\mu$,
		in $\mathcal{M}_+^b(\R^4)$, and denoted by $\mu_m\rightharpoonup\mu$, if $\int fd\mu_m\rightarrow\int fd\mu$, for all $f\in {C}_b(\R^4)$.
		
		\item[(iii)] A sequence $(\mu_m)\subset\mathcal{M}_+^b(\R^4)$ is said to be uniformly tight if, for every $\epsilon>0$, there exists a compact subset $K_\epsilon\subset \R^4$ such that $\mu_m(X\backslash K_\epsilon)\leq \epsilon$ for all $m$. We also say that a set $\mathcal{H}\subset \mathcal{M}_+(\R^4)$ is vaguely bounded if $\sup_{\mu\in\mathcal{H}}\left|\int f d\mu\right|<\infty$, for all $f\in{C}_0(X)$.
	\end{itemize}
\end{definition}

The first lemma below  is a slightly modification of the concentration-compactness lemma  presented in \cite[Lemma I.1]{lions1}. 

\begin{lemma}[Concentration-compactness lemma I] \label{lema35}	
	Suppose that $(\nu_m)$ is a sequence in $\mathcal{M}_+^1(\mathbb{R}^4)$. Then, there is a subsequence, still denoted by $(\nu_m)$, such that one of the following conditions holds:
	\begin{itemize}
		\item[(i)](Vanishing) For all $R>0$ it holds
		$$
		\lim_{m\rightarrow\infty}\left(\sup_{x\in\mathbb{R}^4}\nu_m(B(x,R))\right)=0.
		$$
		
		\item[(ii)](Dichotomy) There is a number $\lambda\in(0,1)$ such that for all $\epsilon>0$ there exists $R>0$ and a sequence $(x_m)$ with the following property: given $R'>R$
		$$
		\nu_m(B(x_m,R))\geq \lambda-\epsilon,
		$$
		$$
		\nu_m(\mathbb{R}^4\backslash B(x_m,R'))\geq 1-\lambda-\epsilon,
		$$
		for $m$ sufficiently large.
		
		\item[(iii)](Compactness) There exists a sequence $(x_m)\subset\mathbb{R}^4$ such that for each $\epsilon>0$ there is a radius $R>0$ with the property
		$$
		\nu_m(B(x_m,R))\geq 1-\epsilon,
		$$
		for all $m$.
	\end{itemize}
\end{lemma}
\begin{proof}
We refer the reader to \cite[Lemma 23]{FM}.
\end{proof}

As already commented, in order to find a minimizer for the minimization problem \eqref{MinNorm}, we will construct a suitable sequence of probability Radon measures and apply Lemma \ref{lema35}. We start with the following.

\begin{lemma}[Concentration-compactness lemma II] \label{lema36}
	 Let $(u_m,w_m)\subset\dot{H}^1(\mathbb{R}^4)\times \dot{H}^1(\mathbb{R}^4)$ be a sequence such that $u_m,w_m\geq 0$ and 
	\begin{equation*}
		\left\{\begin{array}{lcc}
			({u}_m,w_m)\rightharpoonup (u,w), & \hbox{in}& \dot{H}^1(\mathbb{R}^4)\times \dot{H}^1(\mathbb{R}^4),\\
			\mu_m:=(|\nabla u_m|^2+|\nabla w_m|^2) dx \overset{\ast}{\rightharpoonup} \mu,& \hbox{in}& \mathcal{M}_+^b(\mathbb{R}^4)\\
			\nu_m:=\varphi(u_m,w_m) dx \overset{\ast}{\rightharpoonup} \nu,& \hbox{in}& \mathcal{M}_+^b(\mathbb{R}^4). \\
			
		\end{array}\right.
	\end{equation*}
	Then,
	\begin{itemize}
		\item[(i)] There exists an at  most countable set $J$, a family of distinct points $\{x_j\in \mathbb{R}^4;\, j\in J\}$, and a family of non-negative numbers $\{a_j;\, j\in J\}$ such that
		\begin{equation}\label{315}
			\nu=\varphi(u,w)dx+\sum_{j\in J}a_j\delta_{x_j}.
		\end{equation}
		
		\item[(ii)]Moreover, we have
		\begin{equation}\label{316}
			\mu \geq  \left(|\nabla u|^2+|\nabla w|^2\right)dx +\sum_{j\in J}b_j\delta_{x_j}
		\end{equation}
		fore some family $\{b_j;\, j\in J\},b_j>0,$ such that
		\begin{equation}\label{317}
			a_j\leq I^{-2}b_j^2,\quad \forall j\in J.
		\end{equation}
		In particular, $\displaystyle\sum_{j\in J}a_j^{1/2}<\infty$.
	\end{itemize}
\end{lemma}

\begin{proof}
This lemma is inspired in the limit case lemma in \cite{lions2} (see also \cite{evans}, \cite{FM}, and \cite{pastor3} for similar results). The weak converge of $(u_m,w_m)\rightharpoonup (u,w)$ implies that, up to a subsequence, we have $(u_m,w_m)\rightarrow (u,w)$ a.e. in  $\mathbb{R}^4$, which gives $u,w\geq 0$. So, we shall divide the proof into the cases  $(u,w)=(0,0)$ and $(u,w)\neq(0,0)$.

	\textbf{Case 1.} Assume $(u,w)=(0,0)$. 
	
Take $\xi\in C^\infty_0(\mathbb{R}^4)$ and observe that from the vague convergence of $(\nu_m)$, the homogeneity of $\varphi$, and \eqref{DSCN},
	\begin{equation}\label{318}
		\begin{split}
			\int |\xi|^4d\nu=\lim_{m\rightarrow\infty}\int |\xi|^4\varphi(u_m,w_m)dx
			=\lim_{m\rightarrow\infty}\int \varphi(|\xi|u_m,|\xi|w_m)dx
			\leq I^{-2} \liminf_{m\rightarrow \infty}K(\xi u_m,\xi w_m)^2.
		\end{split}
	\end{equation}
	Since $(u_m,w_m)\rightharpoonup (0,0)$ in $\dot{H}^1(\mathbb{R}^4)\times \dot{H}^1(\mathbb{R}^4)$, we know that (see \cite[Theorem 8.6]{lieb}) $(u_m,w_m)\to (0,0)$ in $L^2_{\rm loc}(\mathbb{R}^4)\times L^2_{\rm loc}(\R^4)$. Then,  using the triangular inequality, we obtain
	\begin{equation*}
		\begin{split}
			\Big|\left(\Vert \nabla(\xi u_m)\Vert_{L^2}^2\right.\bigr.&+\bigl.\left.\Vert \nabla(\xi w_m)\Vert_{L^2}^2\right)^{1/2}
			-(\Vert \xi\nabla( u_m)\Vert_{L^2}^2+\Vert \xi\nabla( w_m)\Vert_{L^2}^2)^{1/2}\Big|\\
			&\leq \left(\Vert \nabla(\xi u_m)-\xi\nabla u_m\Vert_{L^2}^2+(\Vert \nabla(\xi w_m)-\xi\nabla w_m\Vert_{L^2}^2\right)^{1/2}\\
			&=\left(\Vert u_m\nabla\xi\Vert_{L^2}^2+\Vert w_m\nabla\xi\Vert_{L^2}^2\right)^{1/2}\\
			&\lesssim \left(\int_{\rm supp(\xi)}|u_m|^2+| w_m|^2\right)^{1/2}\rightarrow 0,\quad m\rightarrow\infty.\\
		\end{split}
	\end{equation*}
	Combining with the vague convergence of $(\mu_m)$, we get
	\begin{equation*}
		\begin{split}
			\liminf_{m\rightarrow\infty}K(\xi u_m,\xi w_m)^2&=\liminf_{m\rightarrow\infty}\left(\int |\xi|^2(|\nabla u_m|^2+|\nabla w_m|^2)dx\right)^2\\
			&=\liminf_{m\rightarrow\infty}\left(\int|\xi|^2d\mu_m\right)^2
			=\left(\int|\xi|^2d\mu\right)^2.
		\end{split}
	\end{equation*}
	Then, from \eqref{318}, we deduce
	\begin{equation}\label{320}
		\int |\xi|^4d\nu\leq I^{-2}\left(\int|\xi|^2d\mu\right)^2,\quad \xi\in C^\infty_0(\mathbb{R}^4).
	\end{equation}
	We claim that \eqref{320} implies that
	\begin{equation}\label{321}
		\nu(E)\leq I^{-2}\mu(E)^2,\quad E\in \mathcal{B}(\mathbb{R}^4),
	\end{equation}
where $\mathcal{B}(\mathbb{R}^4)$ denotes the Borel $\sigma$-algebra on $\R^4$.	Indeed, let $U\subset \mathbb{R}^4$ be an open set and take a compact set $K\subset U$. By $C^\infty$ Urysohn's lemma (see \cite[Lemma 8.18]{Folland}), there exists $g\in C^\infty_0(\mathbb{R}^4)$ obeying $0\leq g\leq 1$, $g=1$ in $K$ and supp$(g)\subset U$. Thus, by \eqref{320},
	$$
	\nu(K)=\int_K g^4d\nu\leq \int g^4d\nu\leq I^{-2}\left(\int g^2d\mu\right)^2\leq I^{-2}\left(\int_{\mbox{\scriptsize supp} (g)}g^2d\mu\right)^2\leq I^{-2}\left(\int_Ud\mu\right)^2,
	$$
that is, $\nu(K)\leq I^{-2}\mu(U)^2$, for all $K\subset U$ compact. Since $\nu$ is a Radon measure, by its inner regularity, we have for all open set $U\subset\mathbb{R}^4$,
	\begin{equation}\label{322}
		\nu(U)\leq I^{-2}\mu(U)^2.
	\end{equation}
	Now, if $E\in\mathcal{B}(\mathbb{R}^4)$ and $U$ is an open set with $E\subset U$, then from \eqref{322}, we get $\nu(E)\leq \nu(U)\leq I^{-2}\mu(U)^2$. Since $\mu$ is a Radon measure, we can use its outer regularity to get \eqref{321}.
	
	Next, consider $D=\{x\in\mathbb{R}^4;\, \mu(\{x\})>0\}$. We may write $D=\displaystyle\bigcup_{k=1}^\infty D_k$, where $D_k=\{x\in\mathbb{R}^4;\,\mu(\{x\})>1/k\}.$ Since $\mu$ is a finite measure we deduce that  $D_k$ is finite for all $k$. Indeed, assume that exists $k_0$ such that $D_{k_0}$ has infinitely many elements, i.e., $D_{k_0}=\{x_j, j\in\mathbb{N}\}$. Then $\mu(D_{k_0})=\sum_{j\in\mathbb{N}}\mu(\{x_j\})>\sum_{j\in\mathbb{N}}1/k_0=\infty$, which contradicts the fact that $\mu$ is finite. Hence, $D_k$  is finite for all $k\in\mathbb{N}$ and the set $D$ is at most countable. Thus,  we may write $D=\{x_j;\,\, j\in J\}$, with $J\subset\mathbb{N}$.
	
	Set $b_j=\mu(\{x_j\})$, $j\in J$, then for any $E\in\mathcal{B}(\mathbb{R}^4)$, we have
	\begin{equation}\label{323}
		\sum_{j\in J} b_j \delta_{x_j} (E)=\sum_{\underset{x_j\in E}{j\in J}}b_j=\sum_{\underset{x_j\in E}{j\in J}}\mu(\{x_j\})\leq \mu(E),
	\end{equation}
	where $\delta_{x_j}(E)=1$, if $x_j\in E$, and $\delta_{x_j}(E)=0$ otherwise. This gives \eqref{316} in the case $(u,w)=(0,0)$. 
	
	Now, observe that from \eqref{321} we have $\nu\ll\mu$ and, by  Radon-Nikodym theorem (see \cite[Section 1.6]{EG}) we may obtain a non-negative function $h\in L^1(\mathbb{R}^4,\mu)$ satisfying
	\begin{equation}\label{324}
		\nu(E)=\int_Eh(x)d\mu(x), \quad E\in\mathcal{B}(\mathbb{R}^4).
	\end{equation} 
	Moreover, $h$ satisfies
	\begin{equation}\label{325}
		h(x)=\lim_{r\rightarrow 0}\frac{\nu(B(x,r))}{\mu(B(x,r))},\quad\mu \mbox{ a.e. }x\in \mathbb{R}^4.
	\end{equation}
	Using \eqref{325} and \eqref{321}, we get $0\leq h(x)\leq I^{-2}\mu(\{x\}).$ Thus, $h(x)=0$, $\mu$ a.e. on $\mathbb{R}^4\backslash D$. In particular, we can rewrite the integral \eqref{324} as
	\begin{equation}\label{326}
		\int_E h(x)d\mu(x)=\sum_{\underset{x_j\in E}{j\in J}}h(x_j)\mu(\{x_j\}).
	\end{equation}
	Setting $a_j=\nu(\{x_j\})$, $j\in J$, we have from \eqref{324} and \eqref{326} that in fact $a_j=h(x_j)b_j$, $\forall j\in J$. Then, for all $E\in\mathcal{B}(\mathbb{R}^4)$, we have
	$$
	\nu(E)=\sum_{\underset{x_j\in E}{j\in J}}h(x_j)\mu(\{x_j\})=\sum_{\underset{x_j\in E}{j\in J}}a_j=\sum_{j\in J}a_j\delta_{x_j}(E),
	$$
	which establishes \eqref{315} for $(u,w)=(0,0)$.
	Finally, inequality \eqref{317} follows immediately from the definitions of $a_j$ and $b_j$ and \eqref{321}. Finally, if we take $E=\R^4$ in \eqref{323} we deduce that $\sum_{j\in J}b_j$ is convergent. Hence, the convergence of the series $\sum_{j\in J} a_j^{1/2}$ follows from \eqref{321}.
	
	\textbf{Step 2.} Case $(u,w)\neq (0,0)$.
	
Let us start by noting that since $u,w\geq 0$ we have $\varphi(u,w)\geq 0$; thus $\varphi(u,w)dx$ defines a positive measure. 
	
We claim the measures 
	\begin{equation*}
		\mu-(|\nabla u|^2+|\nabla w|^2)dx\quad \hbox{and}\quad \nu-\varphi(u,w)dx
	\end{equation*}
	are non-negative.
	
	Indeed, set $(y_m,z_m)=(u_m-u,w_m-w)$ and consider the sequence of measures
	$$
	\tilde{\mu}_m:=(|\nabla y_m|^2+|\nabla z_m|^2)dx\quad \hbox{and}\quad \tilde{\nu}_m:=\varphi(|y_m|,|z_m|)dx.
	$$
	Since $(y_m,z_m)\rightharpoonup (0,0)$ in $\dot{H}^1(\mathbb{R}^4)\times\dot{H}^1(\mathbb{R}^4)$, the sequence $(K(y_m,z_m))$ is uniformly bounded. Hence, since
	$$
	\left| \int fd\tilde{\mu}_m\right|\leq \Vert f\Vert_{L^\infty}K(y_m,z_m),\quad f\in C^\infty_0(\mathbb{R}^4),
	$$
	we have that $(\tilde{\mu}_m)$ is vaguely bounded on $\mathcal{M}_+^b(\mathbb{R}^4)$. Therefore, there are  $\tilde{\mu}\in\mathcal{M}_+^b(\mathbb{R}^4)$ and  a subsequence of $(\tilde{\mu}_m)$, still denoted by $(\tilde{\mu}_m)$,  satisfying (see \cite[Section 31]{Bauer})
	\begin{equation}\label{328}
		\tilde{\mu}_m\overset{\ast}{\rightharpoonup}\tilde{\mu},\quad \mbox{in}\quad \mathcal{M}_+^b(\mathbb{R}^4).
	\end{equation}
	Now, if
	\begin{equation}\label{329}
		\mu_m\overset{\ast}{\rightharpoonup}\tilde{\mu}+(|\nabla u|^2+|\nabla w|^2)dx,\quad \mbox{in}\quad \mathcal{M}_+^b(\mathbb{R}^4),
	\end{equation}
	then by uniqueness of the vague limit, 
	$$
	\mu=\tilde{\mu}+(|\nabla u|^2+|\nabla w|^2)dx,
	$$
	from which we may conclude that $\mu-(|\nabla u|^2+|\nabla w|^2)dx$ is non-negative.
	
	So, we turn our attention to establish \eqref{329}. Since, by assumption, $\partial_{x_i} y_m\rightharpoonup 0$ and  $\partial_{x_i}z_m\rightharpoonup 0$ in $L^2(\mathbb{R}^4)$ and $f\partial_{x_i} u, f\partial_{x_i}w\in L^2(\mathbb{R}^4)$ for each $f\in C_0(\mathbb{R}^4)$, we have
	\begin{equation}\label{330}
		\begin{split}
			\lim_{m\rightarrow\infty}\int f\nabla y_m\cdot \nabla u dx=0 \quad \mbox{and} \quad
			\lim_{m\rightarrow\infty}\int f\nabla z_m\cdot \nabla w dx=0.\\
		\end{split}
	\end{equation}
	Thus, 
	\begin{equation*}
		\begin{split}
			0&\leq \left|\int fd\mu_m-\int f\left[d\tilde{\mu}+(|\nabla u|^2+|\nabla w|^2)dx\right]\right|\\
			&=\left|\int f(|\nabla u_m|^2+|\nabla w_m|^2)dx-\int f\left[d\tilde{\mu}+(|\nabla u|^2+|\nabla w|^2)dx\right]\right|\\
			&=\left|\int f\left[|\nabla y_m|^2+2\nabla y_m\cdot \nabla u+|\nabla u|^2+|\nabla z_m|^2+2\nabla z_m\cdot \nabla w+|\nabla w|^2\right]dx\right.\\
			&\quad \left.-\int fd\tilde{\mu}-\int f(|\nabla u|^2+|\nabla w|^2)dx\right|\\
			&\leq \left|\int fd\tilde{\mu}_m-\int fd\tilde{\mu}\right|+2\left[\left|\int f\nabla y_m\cdot\nabla udx\right|+\left|\int f\nabla z_m\cdot\nabla wdx\right|\right].
		\end{split}
	\end{equation*}
Note the  first term on the right-hand side of the above inequality  goes to zero by \eqref{328}. The second and third ones go to zero in view of and \eqref{330}. Consequently, \eqref{329} holds.
	
Let us now show that $(\tilde{\nu}_m)$ is vaguely bounded in $\mathcal{M}_+^b(\mathbb{R}^4)$. As seen before, $(K(y_m,z_m))$ is uniformly bounded. Then, \eqref{N} yields
	$$
	\left| \int fd\tilde{\nu}_m\right|\leq \Vert f\Vert_{L^\infty}\int \varphi(|y_m|,|z_m|)dx=CN(|y_m|,|z_m|)\leq CK(|y_m|,|z_m|)^2<M,
	$$
	for some constant $M$. Thus we may obtain a subsequence, still denoted by $(\tilde{\nu}_m)$, such that
	\begin{equation*}
		\tilde{\nu}_m\overset{\ast}{\rightharpoonup}\tilde{\nu},\quad \mbox{in}\quad\mathcal{M}_+^b(\mathbb{R}^4),
	\end{equation*}
	
As before, if
	\begin{equation}\label{332}
		\nu_m\overset{\ast}{\rightharpoonup}\tilde{\nu}+\varphi(u,w)dx,\quad \mbox{in}\quad\mathcal{M}_+^b(\mathbb{R}^4),
	\end{equation}
we have $\nu=\tilde{\nu}+\varphi(u,w)dx$ and, consequently, we deduce that $\nu-\varphi(u,w)dx$ is non-negative. So, what is left is to prove \eqref{332}. 
	
	We already know that $\varphi(u,w)\leq C(|u|^4+|w|^4)$. Then, we are able to use Brezis-Lieb's Lemma \ref{BreLie} with $\varphi(|u|,|w|)$ instead of $f$ in the following way: by assumption we may assume $(y_m,z_m)\rightarrow 0$ a.e. in $\mathbb{R}^4$ and, by Sobolev's inequality,  $(u,w)\in L^4(\mathbb{R}^4)\times L^4(\mathbb{R}^4)$ or equivalently $\varphi(|u|,|w|)\in L^1(\mathbb{R}^4)$. Note that 
	$$
	\left|\varphi(|a_1+b_1|,|a_2+b_2|)-\varphi(|b_1|,|b_2|)\right|\leq \epsilon\phi(a_1,a_2)+\psi_\epsilon(b_1,b_2),
	$$
	where $\phi(a_1,a_2)=|a_1|^4+|a_2|^4$ and $\psi_\epsilon(b_1,b_2)\leq C_\epsilon(|b_1|^4+|b_2|^4)$, with $\epsilon>0$. Also, because $(y_m,z_m)$ is uniformly bounded in $L^4(\mathbb{R}^4)$,
	$$
	\int\phi(y_m,z_m)dx\leq M\quad \hbox{and}\quad \int\psi_\epsilon(u,w)dx<\infty,
	$$
	for $M$ independent of $\epsilon$ and $m$. The Brezis-Lieb then Lemma gives us
	\begin{equation}\label{333}
		\lim_{m\rightarrow\infty}\int|\varphi(|u_m|,|w_m|)-\varphi(|y_m|,|z_m|)-\varphi(|u|,|w|)|dx=0.
	\end{equation} 
	Hence, for all $g\in C_0(\mathbb{R}^4),$
	\begin{equation*}
		\begin{split}
			0&\leq \left|\int gd\nu_m-\int g[d\tilde{\nu}+\varphi(u,w)dx]\right|\\
			&=\left|\int g \varphi(u_m,w_m)dx-\int g \varphi(|y_m|,|z_m|)dx+\int g \varphi(|y_m|,|z_m|)dx-\int g[d\tilde{\nu}+\varphi(u,w)]dx\right|\\
			&\leq \Vert g\Vert_{L^\infty}\int |\varphi(|u_m|,|w_m|-\varphi(|y_m|,|z_m|)-\varphi(|u|,|w|)|dx+\left|\int gd\tilde{\nu}_m-\int g\tilde{\nu}\right|.\\
		\end{split}
	\end{equation*}
 By taking the limit as $m\to\infty$, the first term on the right-hand side of the above inequality vanishes in view of \eqref{333}; the second term goes to zero by the vague convergence of $(\tilde{\nu}_m)$. So \eqref{332} holds. This  completes the proof of our claim.

As a consequence of the above claim we have
	\begin{equation*}
		\left\{\begin{array}{lc}
			(|\nabla y_m|^2+|\nabla z_m|^2)dx\overset{\ast}{\rightharpoonup}\mu-(|\nabla u|^2+|\nabla w|^2)dx,&\hbox{ in }\mathcal{M}_+^b(\mathbb{R}^4),\\
			\varphi(|y_m|,|z_m|)dx\overset{\ast}{\rightharpoonup} \nu-\varphi(u,w)dx,&\hbox{ in }\mathcal{M}_+^b(\mathbb{R}^4),\\
		\end{array}\right.
	\end{equation*}
	and we complete the proof of the lemma after applying Step 1. Note, if necessary, in Step 1 we may replace $(\nu_m)$ by  $\nu_m:=\varphi(|u_m|,|w_m|)dx$ in order to get a non-negative measure. This completes the proof of the lemma.
\end{proof}

Before proving Theorem \ref{ESGS}, we will establish an adapted version of Lemma 1.7.4 in \cite{cazenave}, which will help us to avoid the vanishing property. 
\begin{lemma}\label{lema310}
	Let $(u_m,w_m)\subset{L}^4(\mathbb{R}^4)\times {L}^4(\mathbb{R}^4)$ be such that $u_m,w_m\geq 0$ and $\int \varphi(u_m,w_m)dx=1$, for any $m\in\mathbb{N}$. Let $Q_m(R)$ be the concentration function of $\varphi(u_m,w_m)$ defined by
	$$
	Q_m(R):=\sup_{y\in\mathbb{R}^4}\int_{B(y,R)}\varphi(u_m,w_m)dx,\quad R>0.
	$$
	Then, for each $m\in\N$, there is $y=y(m,R)$ such that
	$$
	Q_m(R)=\int_{B(y,R)}\varphi(u_m,w_m)dx. 
	$$
\end{lemma} 

\begin{proof}
	Fix $m\in\mathbb{N}$. By the definition of $Q_m$, for any $R>0$, there is $(y_i)$ in $\mathbb{R}^4$ such that
	$$
	Q_m(R)=\lim_{i\rightarrow\infty}\int_{B(y_{i},R)}\varphi(u_m,w_m)dx>0.
	$$
	Hence, there exists $i_0$ such that if $i>i_0$ then $\int_{B(y_i,R)}\varphi(u_m,w_m)dx\geq \epsilon$, where $\epsilon>0$. 
	
	Let us show that $(y_i)$ is bounded. Otherwise, there is a subsequence, still denoted by $(y_i)$, such that $B(y_j,R)\cap B(y_i,R)=\emptyset$, $\forall i\neq j$. Thus
	$$
	1=\int \varphi(u_m,w_m)dx\geq \sum_{i\geq i_0}\int_{B(y_i,R)}\varphi(u_m,w_m)dx=\infty,
	$$
	which is an absurd. Therefore $(y_j)$ has a convergent subsequence $(y_{j_k})$ with limit $y=y(m,R)$. Applying the dominated convergence theorem, we get
	$$
	Q_m(R)=\lim_{j_k\rightarrow\infty}\int_{B(y_{j_k},R)}\varphi(u_m,w_m)dx=\int_{B(y,R)}\varphi(u_m,w_m)dx,
	$$
	and the proof ie completed.
\end{proof}

\subsection{Proof of Theorem \ref{ESGS}}

Following the strategy in \cite{pastor3}, before proceeding to the proof of Theorem \ref{ESGS}, we first state the following result.
\begin{theorem}\label{TEO311}
	Suppose that $(u_m,w_m)$ is a minimizing sequence for \eqref{MinNorm} with $u_m,w_m\geq 0$. Then, up to translation and dilation $(u_m,w_m)$ is relatively compact in $\mathcal{N}$, that is, there exist a subsequence $(u_{m_j},w_{m_j})$ and sequences $(R_j)\subset\mathbb{R}$, $(y_j)\subset \mathbb{R}^4$ such that the pair $(v_j,z_j)$ given by
	$$
	v_j:=R_j^{-1}u_{m_j}(R_j^{-1}(x-y_j)),\quad  z_j:=R_j^{-1}w_{m_j}(R_j^{-1}(x-y_j)),
	$$
	strongly converges in $\mathcal{N}$ to some $(v,z)$, which minimizes \eqref{MinNorm}.
\end{theorem}

\begin{proof}
By definition of a minimizing sequence for \eqref{MinNorm}, we have
	\begin{equation*}
		\lim_{m\rightarrow\infty}K(u_m,w_m)=I,\quad N(u_m,w_m)=\int \varphi(u_m,w_m)dx=1,\forall m.
	\end{equation*}
In order to make the proof as simple as possible, we will proceed into 6 steps.\\
	\textbf{Step 1.} There exist sequences $(R_m)$ in $\mathbb{R}$ and $(y_m)$ in $\mathbb{R}^4$ such that
	\begin{equation*}
		v_m:=R^{-1}u_{m}(R_m^{-1}(x-y_m)),\quad z_m:=R^{-1}w_{m}(R_m^{-1}(x-y_m))
	\end{equation*}
	satisfies
	\begin{equation}\label{341}
		\sup_{y\in\mathbb{R}^4}\int_{B(y,1)}\varphi(v_m,z_m)dx=\int_{B(0,1)}\varphi(v_m,z_m)dx=\frac{1}{2}.
	\end{equation}
	To show that, let us take $ R>0,s\in\mathbb{R}^4$ and consider the following scaling
	$$
	v_m^{R,s}:=R^{-1}u_m(R^{-1}(x-s)),\quad z_m^{R,s}:=R^{-1}w_m(R^{-1}(x-s)).
	$$
	From \eqref{obs34ii} we get $K(v_m^{R,s},z_m^{R,s})=K(u_m,w_m)$ and $N(v_m^{R,s},z_m^{R,s})=N(u_m,w_m)=1$. Let us consider the concentration function corresponding to  $\varphi(v_m,z_m)$ given by
	$$
	Q_m^{R,s}(t)=\sup_{y\in\mathbb{R}^4}\int_{B(y,t)}\varphi(v_m^{R,s}(x),z_m^{R,s}(x))dx.
	$$
	A change of variables gives $Q_m(t/R)=Q_m^{R,s}(t)$ for all $t\geq 0$ and $s\in\mathbb{R}^4$, where $Q_m$ is defined as in Lemma \ref{lema310}. In particular, for all $m$, $Q_m$ is a non-decreasing function with $Q_m(0)=0$, $Q_m(1/R)=Q_m^{R,s}(1)$ and $Q_m(t)\rightarrow 1$ as $t\rightarrow\infty$. Therefore,
	$$
	\lim_{R\rightarrow 0^+}Q_m^{R,s}(1)=\lim_{R\rightarrow 0^+}Q_m(1/R)=1.
	$$
	Consequently, for any $m$ we can find $R_m>0$ obeying
	\begin{equation*}
		Q_m^{R_m,s}(1)=Q_m(1/R_m)=\frac{1}{2},\quad \forall s\in\mathbb{R}^4,
	\end{equation*}
	i.e., 
	\begin{equation}\label{343}
		\sup_{y\in\mathbb{R}^4}\int_{B(y,1)}\varphi(v_m^{R_m,s},z_m^{R_m,s})dx=Q_m^{R_m,s}(1)=\frac{1}{2},\quad \forall s\in\mathbb{R}^4.
	\end{equation}
	
	On the other hand, since $\int \varphi(v_m^{R_m,s},z_m^{R_m,s})dx=1$ and $v_m^{R_m,s},z_m^{R_m,s}\geq 0$, Lemma \ref{lema310} gives us $y_m\in\mathbb{R}^4$ obeying
	\begin{equation*}
		\begin{split}
			\sup_{y\in\mathbb{R}^4}\int_{B(y,1)}&\varphi(v_m^{R_m,s}(x),z_m^{R_m,s}(x))dx=\int_{B(y_m,1)}\varphi(v_m^{R_m,s}(x),z_m^{R_m,s}(x))dx\\
			&=\int_{B(0,1)}\varphi(R_m^{-1}u_m(R^{-1}_m(r+y_m-s)),R_m^{-1}w_m(R_m^{-1}(r+y_m-s)))dr,\\
		\end{split}
	\end{equation*}
	where we used the change of variables $x=r+y_m$. Choosing $s=2y_m$ and using \eqref{343}, we get
	\begin{equation*}
		\begin{split}
			\int_{B(0,1)}&\varphi(R_m^{-1}u_m(R_m^{-1}(r-y_m),R_m^{-1}w_m(R_m^{-1}(r-y_m))dr\\
			&=\sup_{y\in\mathbb{R}^4}\int_{B(y,1)}\varphi(R_m^{-1}u_m(R_m^{-1}(x-2y_m)),R_m^{-1}w_m(R_m^{-1}(x-2y_m)))dx\\
			&=Q_m^{R_m,2y_m}(1)\\
			&=\frac{1}{2},
		\end{split}
	\end{equation*}
	which is the second equality in \eqref{341}. For the first one, observe that
	\begin{equation*}
		\begin{split}
			\sup_{y\in\mathbb{R}^4}\int_{B(y,1)}\varphi(v_m,z_m)dx&=\sup_{y\in\mathbb{R}^4}\int_{B(y,1)}\varphi(R_m^{-1}u_m(R_m^{-1}(x-y_m)),R_m^{-1}w_m(R_m^{-1}(x-y_m)))dx\\
			&=\sup_{y\in\mathbb{R}^4}\int_{B(y,1)}\varphi(v_m^{R_m,y_m},z_m^{R_m,y_m})dx\\
			&=\frac{1}{2},\\
		\end{split}
	\end{equation*}
	where in the last equality we used \eqref{343}.
	
	Next, from \eqref{obs34ii} and Step 1, $(v_m,z_m)$ is also a minimizing sequence for \eqref{MinNorm} with $v_m,z_m\geq 0$, that is, 
	\begin{equation}\label{344}
		\lim_{m\rightarrow\infty} K(v_m,z_m)=I,\quad N(v_m,z_m)=\int \varphi(v_m,z_m)dx=1,\forall m\in\mathbb{N}.
	\end{equation}
	Particularly, $(v_m,z_m)$ is uniformly bounded in $\mathcal{N}$. Thus, there exist $(v,z)\in\dot{H}^1(\mathbb{R}^4)\times\dot{H}^1(\mathbb{R}^4)$ such that, up to a subsequence, 
	\begin{equation}\label{345}
		(v_m,z_m)\rightharpoonup (v,z)\quad \mbox{in }  \dot{H}^1(\mathbb{R}^4)\times\dot{H}^1(\mathbb{R}^4).
	\end{equation}
	Let us show that $(v_m,z_m)\rightarrow (v,z)$ in $\mathcal{N}$ and $(v,z)$ is a minimizer for \eqref{MinNorm}. Indeed, as in the proof of Lemma \ref{lema36} we have $v,z\geq 0$. Set the sequence of measures 
	\begin{equation*}
		\mu_m:=(|\nabla v_m|^2+|\nabla z_m|^2)dx \quad \mbox{and} \quad \nu_m:=\varphi(v_m,z_m)dx.
	\end{equation*}
	The identity in \eqref{344} gives that $(\nu_m)$ is a  sequence of probability measures. Thus, by Lemma \ref{lema35}, up to a subsequence, occurs one of the following cases: vanishing, dichotomy or compactness. The idea now is to exclude the vanishing and dichotomy cases.
	
	\textbf{Step 2} Vanishing does not occur.
	
	Indeed, in view of \eqref{341} it follows that for $R=1$
	$$
	\lim_{m\rightarrow\infty}\sup_{y\in\mathbb{R}^4}\nu_m(B(y,1))\geq\frac{1}{2}.
	$$
	
	\textbf{Step 3.} Dichotomy does not occur.
	
	Suppose the opposite. Then, there is $\lambda\in(0,1)$ such that for all $\epsilon>0$, there exist $R>0$ and a sequence $(x_m)$ in $\mathbb{R}^4$ such that given $R'>R$ and $m$ sufficiently large,
	\begin{equation}\label{347}
		\nu_m(B(x_m,R)\geq \lambda-\epsilon,\quad \nu_m(\mathbb{R}^4\backslash B(x_m,R'))\geq 1-\lambda-\epsilon.
	\end{equation}
	Thus, for $m$ sufficiently large, fixing $\delta>0$, Corollary \ref{corol39} yields that choosing $\rho$ satisfying $R<\rho<R'$ with $\rho/R'\leq C(\delta)$ and $R/\rho\leq C(\delta)$,
	$$
	\int_{B(x_m,R)}\varphi(v_m,z_m)dx\leq I^{-2}\left[\int_{B(x_m,\rho)}|\nabla v_m|^2+|\nabla z_m|^2dx+\delta K(v_m,z_m)\right]^{2}
	$$
	and
	$$
	\int_{\mathbb{R}^4\backslash B(x_m,R')}\varphi(v_m,z_m)dy\leq I^{-2}\left[\int_{\mathbb{R}^4\backslash B(x_m,\rho)}|\nabla v_m|^2+|\nabla z_m|^2dy+(2\delta +\delta^2) K(u,w)\right]^{2}.
	$$
	Combining both inequalities with \eqref{347}, we get
	\begin{equation}\label{348}
		I\left[(\lambda-\epsilon)^{1/2}+(1-\lambda-\epsilon)^{1/2}\right]\leq K(v_m,z_m)+(3\delta+\delta^2)K(v_m,z_m).
	\end{equation}
	From \eqref{344} the right-hand side of \eqref{348} is bounded by $K(v_m,z_m)+(3\delta+\delta^2)M$, where $M>0$ does not depend on $m$. Therefore, taking $\delta,\epsilon\rightarrow 0$ and $m\rightarrow\infty$ leads to
$I[\lambda^{1/2}+(1-\lambda)^{1/2}]\leq I,$
	that is, $\lambda^{1/2}+(1-\lambda)^{1/2}\leq 1$. But this contradicts the fact that if $\lambda\in(0,1)$ then  $\lambda^{1/2}+(1-\lambda)^{1/2}>1$. Hence, dichotomy does not occurs.
	
	Thereby, Lemma \ref{lema35} implies that compactness occurs, that is, there is a sequence $(x_m)$ in $\mathbb{R}^4$ such that for all $\epsilon>0$ there is a radius $R>0$ such that
	\begin{equation}\label{350}
		\nu_m(B(x_m,R))\geq 1-\epsilon,\quad \forall m.
	\end{equation}
	
	\textbf{Step 4.} The sequence $(\nu_m)$ is uniformly tight.
	
	Indeed, we first show that $B(x_m,R)\cap B(0,1)\neq\emptyset$, for all $m$. Otherwise, there is $m_0$ such that $B(x_{m_0},R)\cap B(0,1)=\emptyset$. Taking $\epsilon\in(0,1/2)$ in \eqref{350} leads  to
	$$
	\int_{B(x_{m_0},R)}\varphi({v}_{m_0},w_{m_0})dx>\frac{1}{2}.
	$$
	Combining with \eqref{341}, we have
	$$
	\int \varphi(v_{m_0},w_{m_0})dx\geq \int_{B(x_{m_0},R)}\varphi(v_{m_0},w_{m_0})dx +\int_{B(0,1)}\varphi(v_{m_0},w_{m_0})dx>\frac{1}{2}+\frac{1}{2}=1,
	$$
	which is a contradiction with \eqref{344}. 
	
	Now, since $B(x_m,R)\subset B(0,2R+1)$, for all $m$, \eqref{350} yields
	$$
	\nu_m(B(0,2R+1))\geq 1-\epsilon, \quad\forall m.
	$$
	Then, because $(\nu_m)$ is a sequence of probability measures, 
	$$
	\nu_m\left(\mathbb{R}^4\backslash \overline{B(0,2R+1)}\right)=1-\nu_m(B(0,2R+1))\leq\epsilon,\quad\forall m.
	$$
	that is, $(\nu_m)$ uniformly tight.
	
	\textbf{Step 5.} Up to a subsequence, $(\nu_m)$ weakly converge to some $\nu\in\mathcal{M}_+^1(\mathbb{R}^4)$.
	
	In fact, note that for each $f\in{C}_0(\mathbb{R}^4)$,
	$$
	\left| \int f d\nu_m\right|\leq\Vert f\Vert_{L^\infty}\nu_m(\mathbb{R}^4)=\Vert f\Vert_{L^\infty}<\infty.
	$$
	Hence, by Theorems 31.2 and 30.6 in \cite{Bauer}, there is $\nu\in\mathcal{M}^b_+(\mathbb{R}^4)$ such that, up to a subsequence, $\nu_m\rightharpoonup\nu$ weakly in $\mathcal{M}^b_+(\mathbb{R}^4)$, that is, 
	\begin{equation}\label{351}
		\int fd\nu_m\rightarrow\int fd\nu,\quad \forall f\in{C}_b(\mathbb{R}^4).
	\end{equation}
	In particular, taking $f\equiv 1$, we have
	\begin{equation}\label{352}
		\nu(\mathbb{R}^4)=\lim_{m\rightarrow\infty}\nu_m(\mathbb{R}^4)=1,
	\end{equation}
	which implies that $\nu\in\mathcal{M}^1_+(\mathbb{R}^4).$
	
	Now, since $K(v_m,z_m)$ is uniformly bounded, then  $(\mu_m)$ is vaguely bounded. Therefore, up to a subsequence, there is $\mu\in\mathcal{M}_+^b(\mathbb{R}^4)$ obeying
	\begin{equation}\label{353}
		\mu_m \overset{\ast}{\rightharpoonup}\mu\quad\mbox{in}\quad \mathcal{M}_+^b(\mathbb{R}^4).
	\end{equation}
	Thus, with \eqref{345}, \eqref{351} and \eqref{353} in hand, we can use Lemma \ref{lema36} to get
	\begin{equation}\label{354}
		\mu\geq (|\nabla v|^2+|\nabla z|^2)dx+\sum_{j\in J}b_j\delta_{x_j},\quad \nu=\varphi(v,z)dx+\sum_{j\in J}a_j\delta_{x_j}
	\end{equation}
	for a family $\{x_j\in\mathbb{R}^4;\, j\in J\}$ with $J$ at most countable and $a_j,b_j\geq 0$ satisfying
	\begin{equation}\label{355}
		a_j\leq I^{-2}b_j^2,\quad \forall j\in J
	\end{equation}
	with $\sum_{j\in J}a_j^{1/2}$ convergent. Hence, \eqref{DSCN}, \eqref{352} and \eqref{355} lead to
	\begin{equation}\label{356}
		\begin{split}
			I&=\liminf_{m\rightarrow\infty}\mu_m(\mathbb{R}^4)\geq\mu(\mathbb{R}^4)
			\geq K(v,z)+\sum_{j\in J}b_j
			\geq I\left[N(v,z)^{1/2}+\sum_{j\in J}a_j^{1/2}\right]\\
			&\geq I\left[N(v,z)+\sum_{j\in J}a_j\right]^{1/2}
			=I[\nu(\mathbb{R}^4)]^{1/2}
			=I,
		\end{split}
	\end{equation}
	where we have used that $\lambda\mapsto \lambda^{1/2}$ is a concave function. Then, for all the inequalities in \eqref{356} to be in fact equalities, it is necessary that at most one of the terms $N(v,z)$ or $a_j,\, j\in J$ must be different from zero. 
	
	\textbf{Step 6.} $a_j=0$ for all $j\in J$.
	
	Suppose that there exist $j_0\in J$ such that $a_{j_0}\neq 0$. Then, from \eqref{352} and the decomposition \eqref{354} it follows that $\nu=a_{j_0}\delta_{x_{j_0}}$, and hence
	\begin{equation}\label{357}
		1=\nu(\mathbb{R}^4)=a_{j_0}.
	\end{equation}
	From \eqref{341} we get
	$$
	\frac{1}{2}\geq \int_{B(x_{j_0},1)}\varphi(v_m,z_m)dx=\nu_m(B(x_{j_0},1)),\quad \forall m.
	$$
In view of the weak convergence \eqref{351},
	$$
	\frac{1}{2}\geq\lim_{m\rightarrow\infty}\nu_m(B(x_{j_0},1))=\nu(B(x_{j_0},1))=\int_{B(x_{j_0},1)}d\nu=a_{j_0},
	$$
 which contradicts  \eqref{357}.
	
	With this in hand, we must be in the case $\nu=\varphi(u,v)dx$ and from \eqref{352}, we obtain
	\begin{equation}\label{358}
		N(v,z)=\int \varphi(v,z)dx=1,
	\end{equation}
	which means that $(v,z)\in\mathcal{N}$. 
	
	To show that $(v,z)$ is a minimizer for \eqref{MinNorm}, it remains to guarantee that $K(v,z)=I$. But, from the definition of $I$ and \eqref{358} it follows that $I\leq K(v,z)$. On the other hand, the lower semi-continuity of the weak convergence \eqref{345}, gives 
	$$
	K(v,z)\leq \liminf K(v_m,z_m)=I.
	$$
	Thus $K(v,z)=I$ and $(v_m,z_m)\rightarrow (v,z)$ strongly in $\mathcal{N}$, completing the proof of the theorem. 
\end{proof}

As an immediate consequence of Theorem \ref{TEO311} we have.

\begin{corollary}
	There is $(v,z)\in\mathcal{N}$ satisfying $N(v,z)=1$ and $K(v,z)=C_{\rm opt}^{-1/2}$, where $C_{\rm opt}$ is the optimal constant in the critical Sobolev-type inequality \eqref{N}.
\end{corollary}

We are now in a position to prove Theorem \ref{ESGS}. 

\begin{proof}[Proof of Theorem \ref{ESGS}]
	We start applying Theorem \ref{TEO311} to get a minimizer of \eqref{MinNorm}, which will be denoted by $(v,z)$. By Lagrange's multiplier theorem, there is a constant $\lambda$ such that for any pair $(f,g)\in\dot{H}^1(\mathbb{R}^4)\times\dot{H}^1(\mathbb{R}^4)$ it holds
	\begin{equation*}
		\begin{split}
			&2\int \nabla v\cdot \nabla fdx=\lambda\int\left( \frac{1}{9}v^3+2z^2v+\frac{1}{3}v^2z\right)fdx,\\
			&2\int \nabla z\cdot\nabla gdx= \lambda\int \left(9z^3+2v^2z+\frac{1}{9}v^3\right)gdx.
		\end{split}
	\end{equation*}
	Taking $f=v$ and $g=z$ and summing up both equations we see that  $2K(v,z)=4\lambda N(v,z)=4\lambda$ and consequently,   $\lambda> 0$. Next, by setting $(P_0,Q_0):=\left(\frac{\lambda}{2}\right)^{\frac{1}{2}}(v,z)$ we deduce that $(P_0,Q_0)$ is non-trivial. Let us show that $(P_0,Q_0)$ is indeed a ground state solution for \eqref{SIST3}. Note that
	\begin{equation*}
		\begin{split}
			\int \nabla P_0\cdot\nabla fdx &= \left(\frac{\lambda}{2}\right)^{\frac{1}{2}}\int \nabla v\cdot \nabla fdx\\
			&=\int \left(\frac{\lambda}{2}\right)^{\frac{3}{2}}\left(\frac{1}{9}v^3+2z^2v+\frac{1}{3}v^2z\right)fdx\\
			&=\int \left(\frac{1}{9}P_0^3+2Q_0^2P_0+\frac{1}{3}P_0^2Q_0\right)fdx\\
		\end{split}
	\end{equation*}
	and
	\begin{equation*}
		\begin{split}
			\int \nabla Q_0\cdot\nabla g dx&= \left(\frac{\lambda}{2}\right)^{\frac{1}{2}}\int \nabla z\cdot \nabla gdx\\
			&=\int \left(\frac{\lambda}{2}\right)^{\frac{3}{2}}\left(9z^3+2v^2z+\frac{1}{9}v^3\right)gdx\\
			&=\int \left(9Q_0^3+2P_0^2Q_0+\frac{1}{9}P_0^3\right)gdx.\\
		\end{split}
	\end{equation*}
	Therefore $(P_0,Q_0)$ is a solution of \eqref{SIST3}. Also, since $K$ is homogeneous of degree two,
	$$
	J(P_0,Q_0)=\frac{K(P_0,Q_0)^2}{N(P_0,Q_0)}=\frac{K\left(\left(\frac{\lambda}{2}\right)^{\frac{1}{2}}(v,z)\right)^2}{N\left(\left(\frac{\lambda}{2}\right)^{\frac{1}{2}}(v,z)\right)}=\frac{\left(\frac{\lambda}{2}\right)^{2}}{\left(\frac{\lambda}{2}\right)^{2}}\frac{K(v,z)^2}{N(v,z)}=J(v,z).
	$$
	Since $(v,z)$ is a minimizer \eqref{MinNorm}, Remark \ref{obs34i} yields that it is also a minimizer of \eqref{Min} and in view of the above identity so is $(P_0,Q_0)$. Remark \ref{obs32} now gives that $(P_0,Q_0)$ is a ground state solution of \eqref{SIST3}.
\end{proof}

\begin{corollary}\label{corbest}
	The inequality 
	\begin{equation}\label{360}
		N(u,w)\leq C_{\rm opt}K(u,w)^{2},
	\end{equation}
	holds for all $(u,w)\in\mathcal{N}$, with the optimal constant given by
	\begin{equation}\label{361}
		C_{\rm opt}=\frac{1}{16\mathcal{E}(P,Q)},
	\end{equation}
	where $(P,Q)$ is any ground state solution of \eqref{SIST3} and $\mathcal{E}$ is defined in \eqref{E}.
\end{corollary}
\begin{proof}
For any ground state $(P,Q)$, from Remark \ref{obs32} and \eqref{Min}, we must have 
	$$
	 C_{\rm opt}^{-1}=J(P,Q)=16S(P,Q)=16\mathcal{E}(P,Q),
	$$
	which is the desired.
\end{proof}

\begin{remark}
We do not know any result concerning the uniqueness of ground state solutions for \eqref{SIST3}. However, note that in view of Lemma \ref{equi}, $	C_{\rm opt}$ is indeed a constant which does not depend on the choice of the ground state.
\end{remark}

\section{Blow-up of radially symmetric solutions}\label{blowup}

This section is devoted to prove Theorem \ref{blow4}.
As mentioned  before, we start considering, for  $\phi\in C^\infty(\mathbb{R}^4)$ and $(u,w)$ solution of \eqref{SIST1}, the function
$$
\mathcal{V}(t)=\int \phi(x)(|u|^2+\sigma^2|w|^2)dx.
$$
It is not difficult to see that 
\begin{equation}\label{VV}
	\mathcal{V}'(t)=2\hbox{Im}\int \nabla\phi\cdot(\widebar{u}\nabla u+\sigma\widebar{w}\nabla w)dx-2\hbox{Im}\int \phi\widebar{u}f(u,w)+\sigma\widebar{w}g(u,w))dx.
\end{equation}
As mentioned before, since the second term in \eqref{VV} does not necessarily vanishes, we will follow the ideas presented in \cite{inui} and work with radially symmetric solutions and the function
\begin{equation}\label{R}
	\mathcal{R}(t)=2\hbox{Im}\int\nabla\phi\cdot(\bar{u}\nabla u+\sigma\bar{w}\nabla w)dx
\end{equation}
instead of $\mathcal{V}$. Following the strategy presented in \cite[Lemma 2.9]{kavian}, we may derive $\mathcal{R}$ to obtain
\begin{equation*}
	\begin{split}
		\mathcal{R}'(t)=& 4\sum_{1\leq m,j\leq 4} \hbox{Re}\int \frac{\partial^2\phi}{\partial{x_m} \partial{x_j}}(\partial_{x_j} \bar{u}\partial_{x_m}u+\partial_{x_j}\bar{w}\partial_{x_m} w)dx\\
		&-\int \Delta^2\phi(|u|^2+|w|^2)dx-\hbox{Re}\int \Delta\phi H(u,w)dx,\\
	\end{split}
\end{equation*} 
where $H(u,w):=\bar{u}f(u,w)+\bar{w}g(u,w)$. Now, observe that if  $u_0,w_0$ are radially symmetric, so are the respective solutions $u,w$. Besides, if we also take $\phi$ to be radially symmetric, we can write  $\phi(x)=\phi(|x|)$, $u(x)=(|x|)$ and $w(x)=w(|x|)$. Then, for $r=|x|$, we have
$$
\sum_{1\leq m,j\leq 4}\hbox{Re}\frac{\partial^2\phi}{\partial{x_m} \partial{x_j}}(\partial_{x_j} \bar{u}\partial_{x_m}u+\partial_{x_j}\bar{w}\partial_{x_m} w)=\phi''(r)\left(|\nabla u|^2+|\nabla w|^2\right).
$$
In this case, we may rewrite $\mathcal{R'}$ as
\begin{equation}\label{R'rs}
	\mathcal{R}'(t)=4\int\phi''(|\nabla u|^2+|\nabla w|^2)dx-\int\Delta^2\phi(|u|^2+|w|^2)dx-\hbox{Re}\int\Delta\phi H(u,w)dx.
\end{equation}

Let us introduce the functional
$$
\mathcal{P}(u,w)=\int\left(\frac{1}{36}|u|^4+\frac{9}{4}|w|^4+|u|^2|w|^2+\frac{1}{9}\hbox{Re}(\bar{u}^3w)\right)dx.
$$
Observe that
\begin{equation*}
	H(u,w)=\widebar{u}f(u,w)+\widebar{w}g(u,w)=\frac{1}{9}|u|^4+9|w|^4+4|u|^2|w|^2+\frac{4}{9}\widebar{u}^3w,
\end{equation*}
and consequently
$$
\hbox{Re}\int H(u,w)dx=\int \left(\frac{1}{9}|u|^4+9|w|^4+4|u|^2|w|^2+\frac{4}{9}\hbox{Re}(\widebar{u}^3w)\right)dx=4 \mathcal{P}(u,w).
$$
Now, we introduce the functional
\begin{equation*}
	\tau(u,w)=K(u,w)-4\mathcal{P}(u,w).
\end{equation*}
Such a functional is sometimes called ``Pohozaev'' functional, because it is closed related with the so-called Pohozaev identities (see \cite[Lemma 2.2]{oliveira}).
Using the definitions of the energy \eqref{Energia}  we may rewrite
\begin{equation}\label{411}
	\tau(u,w)=4E(u,w)-K(u,w)-2\int (|u|^2+\mu|w|^2)dx.
\end{equation}

The next result shows that $\tau$ must be strictly negative when evaluated at a solution of \eqref{SIST1}.

\begin{lemma}\label{lema delta}
	Assume that $(u_0,w_0)\in H^1(\mathbb{R}^4)\times H^1(\mathbb{R}^4)$ and let $(u,w)$ be the corresponding solution of \eqref{SIST1} defined in the maximal time interval of existence $I$. If
	\begin{equation}\label{41}
		E(u_0,w_0)<\mathcal{E}(P,Q)
	\end{equation}
	and
	\begin{equation}\label{42}
		K(u_0,w_0)>K(P,Q),
	\end{equation}
	where $(P,Q)$ is any ground state and $\mathcal{E}$ is the energy in \eqref{E}, then there exists $\delta>0$ such that $\tau(u(t),w(t))\leq-\delta<0$, for all $t\in I$.
\end{lemma}
\begin{proof}
	Notice that from the definition of the energy \eqref{E} and \eqref{36} we have
	\begin{equation}\label{412}
		K(P,Q)=4\mathcal{E}(P,Q).
	\end{equation}
	Moreover, using \eqref{360} we get $|\mathcal{P}(u,w)|\leq N(|u|,|w|)\leq C_{\rm opt}K(|u|,|w|)^2\leq C_{\rm opt}K(u,w)^2. $
	Thus, by conservation of the energy
	\begin{equation}\label{413}
		\begin{split}
			K(u,w)&=2E(u_0,w_0)+2\mathcal{P}(u,w)-\int(|u|^2+\mu|w|^2)\\
			&\leq 2E(u_0,w_0)+2|\mathcal{P}(u,w)|\\
			&\leq 2E(u_0,w_0)+2C_{\rm opt}K(u,w)^2.\\
		\end{split}
	\end{equation}
	Therefore, taking $a=2E(u_0,w_0)$, $b=2C_{\rm opt}$ and $q=2$ in Lemma \ref{lema comp1}, we have $\gamma=(4C_{\rm opt})^{-1}$ and $f(r)=2E(u_0,w_0)-r+2C_{\rm opt}r^2$, for $r>0$. Also, setting $G(t)=K(u(t),w(t))$, it follows from \eqref{413} that
	$$
	f\circ G(t)=2E(u_0,w_0)-K(u(t),w(t))+2C_{\rm opt}K(u(t),w(t))^2\geq 0.
	$$  
	In addition, from \eqref{361},
	$$
	a<\left(1-\frac{1}{q}\right)\gamma \Leftrightarrow E(u_0,w_0)<\frac{1}{16C_{\rm opt}}=\mathcal{E}(P,Q),
	$$
	and, from \eqref{412},
	$$
	G(0)>\gamma \Leftrightarrow K(u_0,w_0)>\frac{1}{4C_{\rm opt}}=4\mathcal{E}(P,Q)=K(P,Q).
	$$
	Therefore, in view of \eqref{41} and \eqref{42} we may apply Lemma \ref{lema comp1} to get
	\begin{equation}\label{414}
		K(u(t),w(t))>K(P,Q),\quad\forall t\in I.
	\end{equation}
	The hypothesis \eqref{41} together with the conservation of the energy and \eqref{412} give
	$$
	4E(u(t),w(t))=4E(u_0,w_0)<4\mathcal{E}(P,Q)=K(P,Q)<K(u(t),w(t)),
	$$
	and as a consequence of \eqref{411},
	\begin{equation*}
		\tau(u(t),w(t))<0,\quad t\in I.
	\end{equation*}
	Now, let us show that there is $\theta>0$, such that
	\begin{equation}\label{416}
		\tau(u(t),w(t))<-\theta K(u(t),w(t)),\forall t\in I.
	\end{equation} Indeed, if $E(u_0,w_0)\leq 0$, then we can take $\theta=1$, and by \eqref{411} we have the desired estimate. On the other hand, suppose that $E(u_0,w_0)>0$ and \eqref{416} does not hold. Thus, there exist sequences $(t_m)\subset I$ and $\theta_m\rightarrow 0$ obeying
	$$
	-\theta_mK(u(t_m),w(t_m))\leq \tau(u(t_m),w(t_m))<0,
	$$
	which implies
	\begin{equation*}
		\begin{split}
			E(u(t_m),w(t_m))&=\frac{1}{4}\tau(u(t_m),w(t_m))+\frac{1}{4}K(u(t_m),w(t_m))+\frac{1}{2}\int(|u(t_m)|^2+\mu |w(t_m))|^2\\
			&\geq (1-\theta_m)\frac{1}{4}K(u(t_m),w(t_m)).\\
		\end{split}
	\end{equation*}
	Again, the conservation of the energy, \eqref{412} and \eqref{414} lead to
	\begin{equation*}
		\begin{split}
			E(u_0,w_0)=E(u(t_m),w(t_m))&\geq(1-\theta_m)\frac{1}{4}K(u(t_m),w(t_m))\\
			&>(1-\theta_m)\frac{1}{4}K(P,Q)\\
			&\geq (1-\theta_m)\mathcal{E}(P,Q).\\
		\end{split}
	\end{equation*}
	Taking $m\rightarrow\infty$ we arrive at a contradiction with \eqref{41}. Hence, the result follows from \eqref{414} and \eqref{416} with $\delta=\theta K(P,Q)$.
\end{proof}

In the proof of Theorem \ref{blow4} we will use \eqref{R'rs} with a suitable function $\phi$. Let us start by defining the compactly supported smooth function $\zeta:\R\to\R$ by
$$
\zeta(r)=
\begin{cases}
e^{-\frac{1}{(r-1)(3-r)}}, \quad 1< r< 3,\\
0, \qquad \textrm{otherwise}.
\end{cases}
$$
Next, for $r\geq0$, we define
$$
\chi(r)=
\begin{cases}
r^2, \qquad 0\leq r\leq 1\\
r^2-{\displaystyle \dfrac{1}{m_0}\int_{1}^{r^2}\int_1^t\zeta(s)dsdt}, \quad 1<r<3,\\
9-m_1, \qquad r\geq3,
\end{cases}
$$
with
$$
m_0=\int_1^9\zeta(s)ds \quad \mbox{and} \quad m_1=\dfrac{1}{m_0}\int_1^9\int_1^t\zeta(s)ds.
$$
It is not difficult to see that  $\chi''(r)\leq 2$ and $0\leq \chi'(r)\leq 2r$, for all $ r\geq 0$.

\begin{lemma}\label{lema chi}
	For $x\in\mathbb{R}^4$, we set $r=|x|$.  Given any $R>0$, we define
 $\chi_R(r)=R^2\chi(r/R)$. Then
	\begin{itemize}
		\item[(i)] If $r\leq R$,
		\begin{equation}\label{48}
			\Delta\chi_R(r)=8\quad\mbox{and}\quad\Delta^2\chi_R(r)=0.
		\end{equation}

		\item[(ii)] If $r\geq R$,
		\begin{equation}\label{49}
			\Delta\chi_R(r)\leq C\quad\mbox{and}\quad|\Delta^2\chi_R(r)|\leq\frac{C}{R^2},
		\end{equation}
		where $C$ is a constant independent of $R$.
	\end{itemize}
\end{lemma}
\begin{proof}
	
	$(i)$ Since $r\leq R$ then $\chi_R(r)=r^2$. Hence,
	$$
	\partial_{x_i}\chi_R(r)=\partial_{x_i}(|x|^2)=2x_i\quad\Rightarrow\quad \partial^2_{x_i}\chi_R(r)=2.
	$$
	Thus, $\Delta\chi_R(r)=8$ and $\Delta^2\chi_R(r)=0$.
	
	$(ii)$ A straightforward calculation leads to
	\begin{equation*}
		\frac{\partial^k\chi_R(r)}{\partial r^k}=\frac{\chi^{(k)}(r/R)}{R^{k-2}}.
	\end{equation*}
	So, for $k=0,1,...$ we have
	\begin{equation}\label{chi}
		\left|\frac{\partial^{k}\chi_R(r)}{\partial r^k}\right|\leq \frac{C}{R^{k-2}}.
	\end{equation}
	On the other hand,
	$$
	\partial_{x_i}\chi_R(r)=R^2\partial_{x_i}\chi(|x|/R)=R\frac{x_i}{|x|}\cdot \chi'(r/R)
	$$
	and
	$$
	\partial^2_{x_i}\chi_R(r)=R\left[\frac{|x|^2-x_i^2}{|x|^3}\cdot\chi'(r/R)+\frac{1}{R}\frac{x_i^2}{|x|^2}\cdot\chi''(r/R)\right].
	$$
	Therefore, 
	$$
	\Delta\chi_R(r)=\frac{3}{r}\frac{\partial\chi_R(r)}{\partial r}+\frac{\partial^2\chi_R(r)}{\partial r^2}.
	$$
	A new calculation yields
	\begin{equation*}
		\Delta^2\chi_R(r)=\frac{\partial^4\chi_R(r)}{\partial r^4}+\frac{6}{r}\frac{\partial^3\chi_R(r)}{\partial r^3}
		-\frac{3}{r^2}\frac{\partial^2\chi_R(r)}{\partial r^2}+\frac{3}{r^3}\frac{\partial\chi_R(r)}{\partial r}.
	\end{equation*}
	Hence, using \eqref{chi} and the fact that $1/r\leq 1/R$, we obtain
	$$
	\Delta \chi_R(r)\leq C\quad \hbox{and}\quad|\Delta ^2\chi_R(r)|\leq\frac{C}{R^2}.
	$$
	The proof of the lemma is thus completed.
\end{proof}

Now, we are in a position to prove Theorem \ref{blow4}.

\begin{proof}[Proof of Theorem \ref{blow4}]
Assume the maximal interval of existence is of the form	 $I=(T_*,T^*)$ and let us prove that both $T^*$ and $T_*$ are finite. Actually, we will focus in the case $T^*<\infty$, because the  argument for $T_*$ follows similarly. Suppose by contradiction that $T^*=\infty$. Taking $\phi(x)=\chi_R(|x|)$, defined as in Lemma \ref{lema chi},  we obtain, from \eqref{R} and \eqref{R'rs},
	$$
	\mathcal{R}(t)=2\hbox{Im}\int\nabla\chi_R\cdot(\nabla u \bar{u}+\sigma\nabla w\bar{w})dx
	$$
	and
	\begin{equation*}
		\begin{split}
			\mathcal{R}'(t)&=8\tau(u,w)+4\int(\chi_R''-2)(|\nabla u|^2+|\nabla w|^2)dx-\int\Delta^2\chi_R(|u|^2+|w|^2)dx\\
			&\quad-\hbox{Re}\int(\Delta\chi_R-8)H(u,w)dx\\
			&=:8\tau(u,w)+\mathcal{R}_1(t)+\mathcal{R}_2(t)+\mathcal{R}_3(t).
		\end{split}
	\end{equation*}
	
	Since by construction $\chi_R''\leq 2$, for any $r\geq 0$, we must have $\mathcal{R}_1\leq 0$. Now, using conservation of the mass \eqref{M} and \eqref{48}-\eqref{49}, we get
	$$
	\mathcal{R}_2(t)\leq\int|\Delta^2\chi_R|(|u|^2+|w|^2)dx\leq CR^{-2}\int_{\{|x|\geq R\}}(|u|^2+|w|^2)dx\leq CR^{-2}M(u_0,w_0).
	$$
	Also, from  \eqref{48},
	\begin{equation*}
		\begin{split}
			\mathcal{R}_3&=-\hbox{Re}\int_{\{|x|\geq R\}}(\Delta\chi_R-8)H(u,w)dx\\
			&\leq C\int_{\{|x|\geq R\}}|\mathrm{Re} H(u,w)|dx\\
			&\leq C\int_{\{|x|\geq R\}}(|u|^4+|w|^4)dx\\
			&=C(\Vert u\Vert_{L^4(|x|\geq R)}^4+\Vert w\Vert_{L^4(|x|\geq R)}^4).
		\end{split}
	\end{equation*}
Next we recall (see, for instance, \cite[page 323]{ogawa}), that for $f\in H^1(\mathbb{R}^4)$ radially symmetric, it holds the radial Gagliardo-Nirenberg inequality
	$$
	\int_{\{|x|\geq R\}}|f|^4\leq CR^{-3}\Vert f\Vert^3_{L^2(|x|\geq R)}\Vert\nabla f\Vert_{L^2(|x|\geq R)}^{1/2}.
	$$
	Then, by Young's inequality, for $\epsilon>0$, we obtain
	\begin{equation*}
		\begin{split}
			\mathcal{R}_3&\leq CR^{-3}(\Vert u\Vert^3_{L^2(|x|\geq R)}\Vert\nabla u\Vert_{L^2(|x|\geq R)}^{1/2}+\Vert w\Vert^3_{L^2(|x|\geq R)}\Vert\nabla w\Vert_{L^2(|x|\geq R)}^{1/2})\\
			&\leq C_\epsilon R^{-4}(\Vert u\Vert^4_{L^2(|x|\geq R)}+\Vert w\Vert^4_{L^2(|x|\geq R)})+\epsilon K(u,w)\\
			&\leq C_\epsilon R^{-4}M(u_0,w_0)^4+\epsilon K(u,w),\\
		\end{split}
	\end{equation*}
	where $C_\epsilon$ is a constant depending on $\epsilon$. Since, from \eqref{411},
	\begin{equation*}
		\epsilon K(u,w)\leq -\epsilon\tau(u,w)+4\epsilon E(u_0,w_0),
	\end{equation*}
	collecting all above estimates, we obtain
	\begin{equation}\label{417}
		\mathcal{R}'(t)\leq (8-\epsilon)\tau(u,w)+CR^{-2}M(u_0,w_0)+C_\epsilon R^{-4}M(u_0,w_0)^2+4\epsilon E(u_0,w_0), \quad \epsilon>0.
	\end{equation}
	Therefore, for $\epsilon\in(0,1)$, Lemma \ref{lema delta} yields
	\begin{equation*}
		\mathcal{R}'(t)\leq -(8-\epsilon)\delta+CR^{-2}M(u_0,w_0)+C_\epsilon R^{-4}M(u_0,w_0)^2+4\epsilon E(u_0,w_0).
	\end{equation*}
	Hence, fixing $R$ as large as necessary and $\epsilon$ as small as necessary, we get  $\mathcal{R}'(t)\leq -2\delta$. Integrating in $[0,t)$, we obtain
	\begin{equation}\label{418}
		\mathcal{R}(t)\leq -2\delta t+\mathcal{R}(0).
	\end{equation}
	On the other hand, by Hölder's inequality,
	\begin{equation}\label{419}
		\begin{split}
			|\mathcal{R}(t)| &\leq 2R\int |\chi'(|x|/R)|(|\nabla u||u|+\sigma|\nabla w||w|)dx\\
			&\leq CR(\Vert u\Vert_{L^2}\Vert \nabla u\Vert_{L^2}+\Vert w\Vert_{L^2}\Vert \nabla w\Vert_{L^2})\\
			&\leq CR M(u_0,w_0)^{1/2}K(u,w)^{1/2}.\\
		\end{split}
	\end{equation}
	Taking $T_0$ sufficiently large such that $\mathcal{R}(0)/\delta <T_0$, by \eqref{418}, we get
	\begin{equation}\label{420}
		\mathcal{R}(t)\leq-\delta t<0,\quad t\geq T_0.
	\end{equation}
	Consequently, \eqref{419} and \eqref{420} imply
	$$\delta t\leq-\mathcal{R}(t)=|\mathcal{R}(t)|\leq CRM(u_0,w_0)^{1/2}K(u,w)^{1/2},$$ that is, for some positive constant $C_0$, 
	\begin{equation}\label{421}
		K(u(t),w(t))\geq C_0t^2,\quad t\geq T_0.
	\end{equation}
	Now, since $\epsilon$ can be chosen arbitrarily small, from \eqref{417} and \eqref{411}, we deduce that
	\begin{equation}\label{422}
		\mathcal{R}'(t)\leq 32E(u_0,w_0)-8K(u,w)+CR^{-2}M(u_0,w_0)+CR^{-4}M(u_0,w_0)^2,
	\end{equation}
	where we have used the energy conservation once again. Now, we may choose $T_1>T_0$, so that 
	$$
	C_04T_1^2\geq 32E(u_0,w_0)+CR^{-2}M(u_0,w_0)+CR^{-4}M(u_0,w_0)^2.
	$$
	Then, from \eqref{421} and \eqref{422} we arrive at
	$$
	\mathcal{R}'(t)\leq -4K(u(t),w(t)), \quad t>T_1.
	$$
	Hence, integrating in $[T_1,t)$, 
	$$
	\mathcal{R}(t)\leq -4\int_{T_1}^tK(u(s),w(s))ds,
	$$
	and combining with \eqref{419}, leads to
	\begin{equation}\label{423}
		4\int_{T_1}^tK(u(s),w(s))ds\leq -\mathcal{R}(t)\leq |\mathcal{R}(t)|\leq CRM(u_0,w_0)^{1/2}K(u(t),w(t))^{1/2}.
	\end{equation}
	Setting $\eta(t):=\int_{T_1}^tK(u(s),w(s))ds$ and $A:=\frac{16}{C^2R^2M(u_0,w_0)} $, we may write
	$$
	A\leq\frac{\eta'(t)}{\eta^2(t)} ,
	$$
	taking $T'>T_1$ and integrating over $[T',t)$, we get
	$$
	A(t-T')\leq \int_{T'}^t\frac{\eta'(s)}{\eta^2(s)}ds=\frac{1}{\eta(T')}-\frac{1}{\eta(t)}\leq \frac{1}{\eta(T')},
	$$
	that is, 
	$$
	0<\eta(T')\leq \frac{1}{A(t-T')}.
	$$
	Hence, taking the limit as $t\rightarrow\infty$ we derive a contradiction. Therefore, $T^*<\infty$ and the proof is completed.
\end{proof}

\section{Appendix: Blow-up in the case of resonance}\label{appendix}

Throughout the paper until now, we have assumed the parameters $\sigma$ and $\mu$ are arbitrary positive real number. This appendix is dedicated to study the existence of blowing-up solutions in the case $\sigma=3$ and $\mu=9$; this is known as the resonance case and comes from the resonant interaction between the beam of a certain fundamental frequency and its third harmonic. Of course, Theorem \ref{blow4} is valid in this case, but here we will establish the blow-up without the assumption of radial symmetry of the initial data. 

In this case, it is easy to see, if $(u,w)$ is a solution of \eqref{SIST1} then $(\widetilde{u}, \widetilde{w})$ defined by 
$$
\widetilde{u}=e^{it}u, \qquad \widetilde{w}=e^{3it}w
$$
is a solution of
\begin{equation}\label{SIST11}
	\begin{cases}
		iu_t+\Delta u+\left(\dfrac{1}{9}|u|^2+2|w|^2\right)u+\dfrac{1}{3}\bar{u}^2w=0,\\
		i3 w_t+\Delta w+(9|w|^2+2|u|^2)w+\dfrac{1}{9}u^3=0,\\
	\end{cases}
\end{equation}
where we have dropped the tilde in order to simplify notation. So, from now on, we will study system \eqref{SIST11}.

First of all, note that the linear term $u$ and $\mu w$ is absent in system \eqref{SIST11} when compared with \eqref{SIST1}. Hence, the energy now is given by
\begin{equation}\label{Energia1}
	\begin{split}
		E(u,w)&:=\frac{1}{2}\int (|\nabla u|^2 +|\nabla w|^2) -\int\left(\frac{1}{36}|u|^4+\frac{9}{4}|w|^4+|u|^2|w|^2+\frac{1}{9}\hbox{Re}(\bar{u}^3w)\right)\\
		&=\frac{1}{2}K(u,v)-P(u,v).
	\end{split}
\end{equation}
In particular, the energy is defined even for functions that do not belong to $L^2(\R^4)$, which is not the case for the energy associated with system \eqref{SIST1}. This forced us to consider system \eqref{SIST1} in $H^1(\R^4)$ instead of $\dot{H}^1(\R^4)$. Recall also we have used the conservation of the mass in the proof of Theorem \ref{blow4}.

So, our goal here is to obtain the existence of blowing-up solutions in $\dot{H}^1(\R^4)$. This seems to be more natural (in the case of resonance) because now the energy of the ground states (see \eqref{E}) coincides with that one in \eqref{Energia1}, that is,
\begin{equation}\label{ee}
	E(P,Q)=\mathcal{E}(P,Q),
\end{equation}
for any ground state $(P,Q)\in \mathcal{G}$. Note also that system \eqref{SIST3} appears when we are looking for stationary solutions of \eqref{SIST11}.

Let us start by recalling a well-posedness result for \eqref{SIST11}. For $T>0$ and $I=[-T,T]$, let
\begin{equation*}
	Z(I):=C(I; \dot{H}^1(\R^4))\cap L^6(I;\dot{H}^{1,12/5}(\R^4)).
\end{equation*}
 The result is the following.
\begin{theorem}\label{BCL1}
	For any $u_0,w_0\in \dot{H}^1(\mathbb{R}^4)$, there exists $T(u_0,w_0)>0$, such that system \eqref{SIST11} has a unique solution $(u,w)\in Z(I)\times Z(I)$, with $I=[-T(u_0,w_0),T(u_0,w_0)]$.  In addition, the map data-solution is continuous and the following blow-up alternative holds: There exist times $T_*,T^*\in(0,\infty]$ such that the solution can be extended to $(-T_*,T^*)$ and if $T^*<\infty$, then
	\begin{equation*}
		\Vert \nabla u(t)\Vert_{L_t^6([0,T^*];L_x^{12/5})}+\Vert\nabla w(t)\Vert_{L_t^6([0,T^*];L_x^{12/5})}=\infty.
	\end{equation*}
	A similar result holds if $T_*<\infty.$
\end{theorem}
\begin{proof}
The proof follows as in the case of the scalar Schr\"odinger equation. For the existence and continuous dependence we refer the reader to \cite{CAZW}  (see also \cite{KM}, where the authors gave an alternative proof). The blow-up alternative then follows as in \cite[Lemma 2.11]{KM} combined with the Sobolev embedding.
\end{proof}

Our blow-up result reads as follows.

\begin{theorem}\label{blow41}
	Suppose $(u_0,w_0)\in \dot{H}^1(\mathbb{R}^4)\times \dot{H}^1(\mathbb{R}^4)$ and let $(u,w)$ be the corresponding solution of \eqref{SIST11} defined in the maximal time interval of existence, say, $J$. Assume in addition that
	$$
	(u_0,w_0)\in L^2(\R^4,|x|^2dx)\times L^2(\R^4,|x|^2dx).
	$$
	If
	\begin{equation}\label{d1}
		E(u_0,w_0)<{E}(P,Q)
	\end{equation}
and
	\begin{equation}\label{d2}
		K(u_0,w_0)>K(P,Q),
	\end{equation}
	where $(P,Q)$ is any ground state in $\mathcal{G}$, then the time interval $J$ is finite.
\end{theorem} 

The proof of Theorem \ref{blow41}  follows from the convexity method with the help of the following virial identity. 

\begin{proposition}\label{virialprop}
	Assume
$$
u_0,w_0\in \dot{H}^1(\mathbb{R}^4)\cap L^2(\R^4,|x|^2dx).
$$
	and define
	$$
	V(t)=\int |x|^2(|u(t)|^2+9|w(t)|^2)dx,
	$$
	where $(u(t),w(t))$ is the maximal solution of \eqref{SIST11}, with initial data $(u_0,w_0)$, and defined in the maximal time interval J. Then $V\in C^2\left(J\right)$ and
	\begin{equation}\label{vlinha}
		V'(t)=4Im \int \left(\overline{u}(t)x\cdot\nabla u(t)+3 \overline{w}(t)x\cdot\nabla w(t)\right)dx
	\end{equation}
	and 
	\begin{equation}\label{vdoislinha}
		\begin{split}
			V''(t)=32 E(u(t),w(t))-8K(u(t),w(t)).
		\end{split}
	\end{equation}
\end{proposition} 
\begin{proof}
For the proof we refer the reader to \cite[Proposition 4.4]{oliveira}. We emphasize that the energy $E$ here is as defined in \eqref{Energia1}.
\end{proof}

Note that in the proof of Lemma \ref{lema delta},  we have proved that under conditions \eqref{41} and \eqref{42} (which in view of \eqref{ee} are the same assumptions of Theorem \ref{blow41}),
$$
	K(u(t),w(t))>K(P,Q),
$$
for all $t$ as long as the solution exists. Actually, by a continuity argument we can ``improve'' such a inequality as we will see below. To do so, we use the following (continuity) lemma.

\begin{lemma}\label{lema comp11}
	Let $I\subset\mathbb{R}$ be an open interval with $0\in I$, $a\in\R$, $b>0$ and $q>1$. Define $\gamma=(bq)^{-1/(q-1)}$ and $f(r)=a-r+br^q$, for $r>0$. Let $G(t)$ be a non-negative continuous function such that $f\circ G\geq 0$ in $I$. Assume there exists a small $\varepsilon>0$ such that $$a<(1-\varepsilon)\displaystyle\left(1-\displaystyle\frac{1}{q}\displaystyle\right)\gamma \qquad \mbox{and} \qquad G(0)>\gamma.$$
	 Then, there exists $\delta>0$ such that $G(t)>(1+\delta)\gamma$, for all $t\in I$.

\end{lemma}
\begin{proof}
	The proof again follows by continuity and is a consequence of Lemma \ref{lema comp1}. See, for instance, \cite[Corollary 3.2]{pastor2}.
\end{proof}

With Lemma \ref{lema comp11} in hand, we are able to prove Theorem \ref{blow41}.

\begin{proof}[Proof of Theorem \ref{blow41}]
Using \ref{corbest} and \eqref{ee}, as  in \eqref{413} we obtain
	\begin{equation*}
	\begin{split}
		K(u,w)&=2E(u_0,w_0)+2\mathcal{P}(u,w)\\
		&\leq 2E(u_0,w_0)+2C_{\rm opt}K(u,w)^2.\\
		&= 2E(u_0,w_0)+\frac{1}{8E(P,Q)}K(u,w)^2.
	\end{split}
\end{equation*}
In view of \eqref{d1} we may found a small $\varepsilon>0$ such that $$	E(u_0,w_0)<(1-\varepsilon){E}(P,Q).$$
The idea now is to use Lemma \ref{lema comp11} with $a=2E(u_0,w_0)$, $b=(8E(P,Q))^{-1}$, $q=2$, $G(t)=K(u(t),w(t))$ and
$$
\gamma = (bq)^{-1}=4E(P,Q)=K(P,Q),
$$
where we used \eqref{412} in the last equality. Note that
	$$
a<(1-\varepsilon)\left(1-\frac{1}{q}\right)\gamma \Leftrightarrow E(u_0,w_0)<(1-\varepsilon){E}(P,Q)
$$
and
$$
G(0)>\gamma \Leftrightarrow K(u_0,w_0)>K(P,Q).
$$
Therefore, we may obtain a $\delta>0$ such that
$$
K(u(t),w(t))>(1+\delta)K(P,Q), \quad \mbox{for all}\;\; t\in J.
$$
Consequently, from \eqref{vdoislinha},
\[
\begin{split}
V''(t)&= 32 E(u_0,w_0)-8K(u(t),w(t))\\
&<32E(P,Q)-8(1+\delta)K(P,Q)\\
&=8K(P,Q)-8(1+\delta)K(P,Q)=-8\delta K(P,Q).
\end{split}
\]
This means that the second derivative of $V$ is bounded by a negative constant, which implies  that $J$ must be finite.
\end{proof}


\subsection*{Acknowledgment}
M. H. is supported by Universidade Estadual de Campinas, through CAPES-Brazil grant 88887.501354/2020-00. A. P. is partially supported by S\~ao Paulo Research Foundation (FAPESP) grant 2019/02512-5 and
Conselho Nacional de Desenvolvimento Cient\'ifico e Tecnol\'ogico (CNPq) grant 309450/2023-3. The authors would like to thank the referees for helpful comments, which improved the presentation of the paper.

\section*{Data Availability Statement}

No datasets were generated or analyzed during the current study.



\begin{thebibliography}{99}
	
	\bibitem{AP}  Angulo, J.,  Pastor, A.: \emph{Stability of periodic optical solitons for a nonlinear Schrödinger system}, Proc. Roy. Soc. Edinburgh Sect. A, v. 139, n. 5, p. 927--959, (2009).		
	
	\bibitem{ardila} Ardila, A. H., Dinh, V. D., Forcella, L.: \emph{Sharp conditions for scattering and blow-up for a system of
		NLS arising in optical materials with {$\chi^3$} nonlinear response}, Comm. Partial Differential Equations, v. 46, n. 11, p. 2134-2-170, (2021).
	
	\bibitem{Bauer}  Bauer, H.: \emph{Measure and Integration Theory}, De Gruyter Studies in Mathematics, v. 26, Walter de Gruyter \& CO, (2001).	
		
	\bibitem{B}  Bégout, P.: \emph{Necessary conditions and sufficient conditions for global existence in the nonlinear Schrödinger equation}, Adv. Math. Sci. Appl., v. 12, n.02, p. 817--827, (2002)	
	
	\bibitem{BL}  Brezis, H.,  Lieb, E.: \emph{A relation between pointwise convergence of function and convergence of functional}, Proc. Amer. Math. Soc., v. 88, no. 3, 486--490,  (1983).
	

	
	\bibitem{cazenave}  Cazenave, T.: \emph{Semilinear Schrödinger Equations}, Amer. Math. Soc.,  Courant Lecture Notes in Mathematics, v. 10, Providence, RI, (2003).
	
	\bibitem{CAZW} Cazenave, T., Weissler, F. B.: \emph{The Cauchy problem for the critical nonlinear Schrödinger equation in $H^s$},  Nonlinear Anal. Theory Methods Appl., v. 14, n. 10, p. 807--836, (1990).		

\bibitem{colin} Colin, M., Watanabe, T.: \emph{Stable standing waves for a Schrödinger system with 	nonlinear {$\chi^3$} response}, J. Math. Phys., v. 64, n. 10, (2023).

\bibitem{D} Dodson, D.: \emph{Global well-posedness and scattering for the focusing, cubic Schr\"odinger equation in dimension $d=4$}, 	Ann. Sci. Éc. Norm. Supér., v. 52, no. 1, p. 139--180, (2019).

	\bibitem{EG}  Evans, L. C.,  Gariepy, R. F.: \emph{Measure Theory and Fine Properties of Functions}, Studies in Advanced Mathematics, CRC Press, Boca Raton, FL, (1992).
	
	\bibitem{evans}	Evans, L.C.: \emph{Weak Convergence Methods for Nonlinear Partial Differential Equations}. In: CBMS Regional
	Conference Series in Mathematics, vol. 74. American Mathematical Society, Providence (1990).

	\bibitem{FM}  Flucher, M.,  Muller, S.: \emph{Concentration of low energy extremals}, Ann. Inst. H. Poincaré Anal. Non Linéare, v. 16, no. 3, p. 269--298,  (1999).	
	
	\bibitem{Folland}  Folland, G. B.: \emph{Real Analysis: Modern Techniques and Their Applications}, Pure and Applied Mathmatics. John Wiley \& Sons, Inc., New York, second edition, (1999).
	
	\bibitem{inui}   Inui, T.,  Kishimoto, N., Nishimura, K.: \emph{Blow-up of the radially symmetric solution for the quadratic nonlinear Schrödinger system without mass-ressonance}. Nonlinear Analysis, v. 198, Paper 111895, (2020).	
		
	\bibitem{kavian}  Kavian, O.:\emph{A remark on the blowing-up solutions to the Cauchy problem for Nonlinear Schrödinger Equations}, Trans. Am. Math. Soc., v. 299, no. 1, p. 193--203,  (1987).
	
	\bibitem{KM}  Kenig, D. E., Merle, F.: \emph{Global well-posedness, scattering and blow-up for the energy-critical, 	focusing, non-linear Schr\"odinger equation in the radial case}, Invent. Math., v. 166, no. 3,  p. 645--675, (2006).
	
	\bibitem{KV} Killip, K.,  Visan, M.: \emph{The focusing energy-critical nonlinear Schr\"odinger equation in dimensions five and higher},  Amer. J. Math., v. 132, no. 2,  p. 361--424, (2010).
	
\bibitem{likai}	Li, Y., Wang, K.,  Wang, Q.: \emph{A system of NLS arising in optical material without
Galilean symmetry}, arXiv:2401.11683v1 (2024).
	
\bibitem{lieb}	Lieb, E., Loss, M.: Analysis, Graduate Studies in Mathematics, v. 14, 2nd edition. American Mathematical Society, Providence (2001).
	
	\bibitem{linares} Linares, F., Ponce, G.: \emph{Introduction to nonlinear dispersive equations}, Universitext, 2nd Edition, Springer, New York, (2015).
	
	\bibitem{lions1}  Lions, P. L.:\emph{The concentration-compactness principle in the calculus of variations. The locally compact case}, I, Ann. Inst. H. Poincaré Anal. Non Linéaire, v. 1, no 2, p. 109--145, (1984).
	
	\bibitem{lions2} Lions, P. L.: \emph{The concentration-compactness principle in the calculus of variations. The limit case, part 1.} Rev. Mat. Iberoamericana, v. 1, no. 1, p. 145--201,  (1985).
	
	\bibitem{pastor3}  Noguera, N.,  Pastor, A.: \emph{Blow-up solutions for a system of Schrödinger equations with general quadratic-type nonlinearities in dimensions five and six}, Calc. Var. Partial Differential Equations, v. 61, paper 111, (2022).	
	
	\bibitem{oliveira}  Oliveira, F.,  Pastor, A.: \emph{On a Schr\"odinger system arizing in nonlinear optics}, Analysis and Mathematical Physics, v. 11, article number 123, (2021).	

		\bibitem{ogawa}  Ogawa, T.,  Tsutsumi, Y.: \emph{Blow-up of $H^1$ solutions for the Nonlinear Schrödinger equation}, J. Differential Equations, v. 92 , no. 2, p. 317--330, (1991).	
			
	\bibitem{pastor4}  Pastor, A.: \emph{Nonlinear and spectral stablity of periodic traveling wave solutions for a nonlinear Schrödinger system}, Differential Integral Equations, v. 23, n. 1--2, p. 125--154, (2010).
	
	\bibitem{pastor2}  Pastor, A.: \emph{Weak concentration and wave operator for a 3D coupled nonlinear Schrödinger system}, J. Math. Phys., v. 56, 021507--1 to 021507--18, (2015).	
	
	\bibitem{stefanov} Ramadan, A., Stefanov A. G.: \emph{On the stability of solitary waves in the NLS system of the 		third-harmonic generation}, Anal. Math. Phys., v. 14, article number 4, (2024).
	
	\bibitem{SBK} Sammut,  R. A.,  Buryak A. V.,  Kivshar, Y. S.:\emph{Bright and dark solitary waves in
the presence of third-harmonic generation},  J. Opt. Soc. Amer. B, v. 15, n. 05, p. 1488-1496, (1998).
	
	
	\bibitem{T} Talenti, C.: \emph{Best constant in Sobolev inequality}, Ann. Mat. Pura Appl., v. 110, n. 4, p. 353-372, (1976). 

	\bibitem{zhang} Zhang, G., Duan, Y., \emph{Some results for a system of {NLS} arising in optical
		materials with {$\chi^3$} nonlinear response}, Math. Methods Appl. Sci., v. 46, n. 17, p. 18011-18034, (2023).
	
	

\end{thebibliography}
\end{document}